\documentclass[10pt,twoside]{article}
\usepackage{amsmath, amsfonts, amsthm, amssymb, amscd,mathrsfs}

\usepackage{bbm}
\usepackage[most]{tcolorbox}
\usepackage{hyperref}
\makeatletter
\newdimen\rh@wd
\newdimen\rh@hta
\newdimen\rh@htb
\newbox\rh@box
\def\rh@measure#1{\setbox\rh@box=\hbox{$#1$}\rh@wd=\wd\rh@box \rh@hta=\ht\rh@box}

\def\widecheck#1{\rh@measure{#1}%
  \setbox\rh@box=\hbox{$\widehat{\vrule height \rh@hta width\z@ \kern\rh@wd}$}%
  \rh@htb=\ht\rh@box \advance\rh@htb\rh@hta \advance\rh@htb\p@
  \ooalign{$\vrule height \ht\rh@box width\z@ #1$\cr
           \raise\rh@htb\hbox{\scalebox{1}[-1]{\box\rh@box}}\cr}}
\makeatother

\numberwithin{equation}{section}  
\usepackage{perpage} 
\MakePerPage{footnote}  
\oddsidemargin 	.in
\evensidemargin.in
\topmargin    - 0.75in
\numberwithin{equation}{section}
        \newtheorem{theorem}{Theorem}[section]
        \newtheorem{proposition}[theorem]{Proposition}
        \newtheorem{lemma}[theorem]{Lemma}

\let\oldmarginpar\marginpar
\renewcommand\marginpar[1]{\-\oldmarginpar[\raggedleft\footnotesize #1]
{\raggedright\footnotesize #1}}

\newcommand \bei {\begin{itemize}}

\newcommand \eei {\end{itemize}} 
\newcommand \Hcal {\mathcal H}

\newcommand \be {\begin{equation}}
\newcommand \bel {\be\label}
\newcommand \ee {\end{equation}}
\newcommand \la \langle
\newcommand \ra \rangle 
\newcommand	\Kcal 	{\mathcal K}

\newcommand	\RR 		{\mathbb R}

\newcommand \del {{\partial}}
\newcommand \underdel {\underline{\partial}}

\newcommand \eps \epsilon

\newcommand{\myfootnote}[1]{
    \renewcommand{\thefootnote}{}
    \footnotetext{\hspace{-16.5pt}\scriptsize#1}
    \renewcommand{\thefootnote}{\arabic{footnote}}
}
 
\begin{document}

\title{The top-order energy of quasilinear wave equations
\\
 in two space dimensions is uniformly bounded}
 
\author{
Shijie Dong${}^{\,\text{a}}$,
Philippe G. LeFloch${}^{\,\text{b}}$,
and Zhen Lei${}^{\,\text{a}, \text{c}, \ast}$  
}

\date{\today}

\maketitle

\begin{abstract} Alinhac solved a long-standing open problem in 2001 and established that quasilinear wave equations in two space dimensions with quadratic null nonlinearities admit global-in-time solutions, provided that the initial data are compactly supported and sufficiently small in Sobolev norm. In this work, Alinhac obtained an upper bound with polynomial growth in time for the top-order energy of the solutions. A natural question then arises whether the time-growth is a true phenomena, despite the possible conservation of basic energy. 
Analogous problems are also of central importance for Schr\"odinger equations and the incompressible Euler equations in two space dimensions, as studied by Bourgain, Colliander-Keel-Staffilani-Takaoka-Tao, Kiselev-Sverak, and others. 
In the present paper, we establish that the top-order energy of the solutions in Alinhac theorem remains globally bounded in time, which is opposite to Alinhac's blowup-at-infinity conjecture.  
\end{abstract}

\maketitle 

\tableofcontents  
 

\myfootnote{
${}^\ast$Corresponding author.\\
${}^\text{a}$ Fudan University, School of Mathematical Sciences, 220 Handan Road, Shanghai, 200433, China. \\
${}^\text{b}$ Laboratoire Jacques-Louis Lions and Centre National de la Recherche Scientifique,
Sorbonne Universit\'e, 4 Place Jussieu, 75252 Paris, France. Email: contact@philippelefloch.org.  
\\
${}^\text{c}$ Shanghai Center for Mathematical Sciences, Shanghai 200433, China. 
Email: shijiedong1991@hotmail.com, zlei@fudan.edu.cn. 
\\
{\sl AMS :} 35L05 
}


\section{Introduction}

\paragraph{Background.}

We consider quasilinear wave equations in two space dimensions with quadratic null nonlinearities.
In 2001 Alinhac established a global-in-time existence theory \cite{Alinhac1} when the initial data are
 compactly supported and sufficiently small in Sobolev norm, which was a long-standing open problem then.
In this work, Alinhac established an upper bound with polynomial growth (in time) for the top-order energy of the solutions.
A natural question is that whether the time-growth of the top-order energy is a true phenomena or not. In fact, Alinhac conjectured \cite{Alinhac02, Alinhac03, Alinhac-book} that this should be a true phenomena and called it ``blowup-at-infinity'', in the context of three dimensional, quasilinear wave equations enjoying a weak null structure.  
Interestingly, Alinhac conjecture was recently established by Deng and Pusateri \cite{Deng2020} for a class of three-dimensional nonlinear wave equations satisfying the weak null condition, and the authors did prove 
the high-order energy of the solutions cannot be uniformly bounded in time. 
On the other hand, the same question for two-dimensional quasilinear wave equations with null nonlinearities remained 
 open until now. Yet, there has been substantial progress concerning this question in recent years. 
 
In \cite{Cai}, Cai, Lei, and Masmoudi showed the uniform boundedness of the top-order energy for a class 
of two-dimensional quasilinear wave equations with (hidden) strong null structure in the sense of Lei \cite{Lei16}. 
More recently, Cai \cite{Cai2020} considered the system of two-dimensional incompressible elastodynamics, and again proved 
the uniform boundedness of the top-order energy. 
There are also some related results for three-dimensional wave equations.
For instance, Wang \cite{Wang2014} proved that quasilinear wave equations with quadratic null nonlinearities in $\RR^{1+3}$ admit compactly supported small global-in-time solutions with uniformly bounded energy.
Lei and Wang in \cite{Lei-Wang} showed the uniform boundedness of the top-order energy for the three dimensional incompressible isotropic elastodynamics. Introducing the ``hyperboloidal foliation method'', LeFloch and Ma \cite{PLF-YM-book} treated nonlinear systems coupling wave equations and Klein-Gordon equations and prove a uniform energy bound, 
 while LeFloch and Wei \cite{PLF-CW} treated the problem of evolving relativistic membranes.  


\paragraph{Other models.}

In addition to wave equations, the study of the time-growth behavior of the (high-order) energy is also of central importance in the theory of the Schr\"odinger equations, two-dimensional incompressible Euler equations, as well as Hamiltonian systems. The time-growth of the top-order energy reflects the {\sl cascade of energy} to its high frequency part. 
In \cite{Bourgain96} (and \cite{Bourgain99}, respectively), Bourgain constructed a wave equation (Schr\"odinger equation, resp.) whose solutions have unbounded high-order energy, and later in \cite{Bourgain00} raised a conjecture concerning
 the time-growth of high-order Sobolev norms for Schr\"odinger equations. 
 Later on, Colliander, Keel, Staffilani, Takaoka, and Tao \cite{Tao2010} proved that the $H^s$--energy (with $s>1$) of 
 the cubic defocusing nonlinear Schr\"odinger equation posed on the two dimensional torus 
 has the property of generating high-frequency modes within a certain bounded time interval. This result
 can be regarded as a partial answer to Bourgain conjecture in \cite{Bourgain00}.
Recently, in the study of the two-dimensional incompressible Euler equations posed on a disk, Kislev and Sverak \cite{Sverak} constructed initial data such that the vorticity gradient grows as a double exponential in time.
 (See also Lei and Shi \cite{Lei-Shi} for a result on the torus and an exponential growth.)
  Various other interesting results related to this topic can be found in the aforementioned references as well as the references cited therein.


\paragraph{Main result.}

In the present paper, we establish that the top-order energy of the solutions to \eqref{eq:model-2d} 
remains {\sl globally bounded in time.} This is achieved by applying the hyperboloidal foliation method 
and taking advantage of the null structure in order to uncover extra decay in the hyperbolic time $s = \sqrt{t^2-|x|^2}$. We emphasize that such a phenomena substantially differs from that of Alinhac's blowup-at-infinity conjecture \cite{Alinhac02, Alinhac03, Alinhac-book, Deng2020}. 

We are interested in the quasilinear wave equation
\bel{eq:model-2d}
\aligned
- \Box w + P^{\gamma \alpha \beta} \del_\gamma w \del_\alpha \del_\beta w = 0,
\endaligned
\ee
when compactly supported initial data are prescribed on a constant time slice $t = t_0 = 2$: 
\bel{eq:ID-2d}
(w, \del_t w)(t_0, \cdot) = (w_{0}, w_{1}),
\ee
the data $w_0, w_1$ having Sobolev regularity (see below).
The wave operator is defined as $\Box = m^{\alpha \beta} \del_\alpha \del_\beta = -\del_t \del_t + \del_a \del^a$ with $m = \text{diag}(-1, 1, 1)$, while for simplicity in the presentation $P^{\gamma \alpha \beta}$ are constants satisfying the standard null condition, i.e.
$P^{\gamma \alpha \beta} \xi_\gamma \xi_\alpha \xi_\beta = 0$ for all 
$ \xi_0^2 = \xi_1^2 + \xi^2_2$ and, in addition, the symmetry condition
$P^{\gamma \alpha \beta} = P^{\gamma \beta \alpha}$.
While we focus here on a model problem, our method should be applicable well beyond \eqref{eq:model-2d}. 

Latin letters $a, b, \ldots \in \{ 1, 2\}$ are used for spatial indices and Greek letters $\alpha, \beta, \ldots \in \{0, 1, 2\}$ represent spacetime indices, while Einstein summation convention is adopted unless otherwise specified.
Without loss of generality, the initial data $(w_{0}, w_{1})$ are assumed to be supported in a ball, say 
$\{(t, x) : t  = 2, |x| \leq 1 \}$, so that the solution is supported within the region $\{ (t, x) : t \geq |x| + 1, t \geq 2 \}$.

Our main result is as follows, and provides us with a uniform control of the energy associated with any vector field in the following list, referred to as admissible vector fields, 
\be
\Gamma \in \{\del_\alpha, L_a = x_a \del_t + t \del_a, \Omega_{ab} = x_a \del_b - x_b \del_a, L_0 = t\del_t + x^a \del_a\}. 
\ee

\begin{theorem}\label{thm:main}
Let $N \geq 4$ be an integer. 
Consider the wave equation \eqref{eq:model-2d} together with initial data $(w_{0}, w_{1})$ on the time slice $t = 2$ supported in the ball $\{x : |x| \leq 1\}$. 
For any $\delta>0$ there exists $\epsilon_0 > 0$ such that, for all $\eps \leq \eps_0$ and all initial data satisfying
\be 
\| w_{0} \|_{H^{N+1}(\RR^2)} + \| w_{1} \|_{H^N(\RR^2)} 
\leq \eps,
\ee
the Cauchy problem \eqref{eq:model-2d}--\eqref{eq:ID-2d} admits a global-in-time solution $w$, which
 decays according to 
\bel{eq:solution-decay}
|w(t, x)| \lesssim t^{-1/2+\delta/2},
\qquad
|\del w(t, x)| \lesssim t^{-1/2}.
\ee
In addition, the total energy of this solution is \textbf{uniformly bounded}, i.e.~for any admissible field  $\Gamma$ 
\bel{eq:top-bound} 
\| \del \Gamma^I w \|_{L^2(\RR^2)} \lesssim 1,
\qquad
|I| \leq N. 
\ee
\end{theorem}

The global existence part in this theorem has been established by Alinhac using the ghost weight energy method in \cite{Alinhac1}. Here the improvement concerns the uniform boundedness of the top-order energy as stated in \eqref{eq:top-bound} (together with the decay rate for the solution itself in \eqref{eq:solution-decay}).
Our strategy is first to show the energy for the restriction of the solution on a hyperboloidal foliation is uniformly bounded, and then to deduce that the energy on constant time $t$ slices is also uniformly bounded. 
The analysis is based on the Hyperboloidal Foliation Method developed in LeFloch and Ma \cite{PLF-YM-book} for coupled systems of wave and Klein-Gordon equations. As far as either (uncoupled) wave equations or Klein-Gordon equations are concerned this strategy was investigated first 
in pioneering work by Klainerman \cite{Klainerman85} and H\"ormander \cite{Hormander}.


\paragraph{Brief history on related topics.}

In seminal work by Klainerman \cite{Klainerman86} and Christodoulou \cite{Christodoulou}, the wave equation with null nonlinearities and sufficiently small initial data was shown to admit global-in-time solutions in $\RR^{1+3}$. However due to the slow decay of linear waves in $\RR^{1+2}$, nonlinear wave equations in $\RR^{1+2}$ are somewhat more difficult to handle. In the framework of Klainerman’s vector field method, Alinhac found a class of ``good derivatives'' and proved a new kind of energy estimate, which is called the ``ghost weight'' energy estimate. Based on this idea, Alinhac \cite{Alinhac1}
succeeded to prove that quasilinear wave equations with null nonlinearities admit small global-in-time solutions in $\RR^{1+2}$. 
In \cite{Alinhac1}, the top-order energy to the solutions grows polynomially in time. 
Whether or not the time-growth on this two-dimensional problem is a true phenomena as stated 
in Alinhac's blowup-at-infinity conjecture \cite{Alinhac02, Alinhac03, Alinhac-book, Deng2020} remains an open question until now.

Following Alinhac's pioneering work \cite{Alinhac1} on two-dimensional quasilinear null wave equations, several interesting advances were also made in recent years. In \cite{Yin} concerning the equation \eqref{eq:model-2d}, the authors removed the compactness assumption on the initial data and, to this end, relied on a class of weighted $L^\infty$--$L^\infty$ estimates. 
A similar result for two-dimensional, fully nonlinear, wave equations satisfying the null condition was obtained in \cite{Cai}, and this result was achieved by relying on an inherent strong null structure \cite{Lei16} enjoyed by nonlinearities. Another interesting work is \cite{He2020}, in which a detailed description of the scattering properties of solutions was derived.


Next we turn to the description of various progress related to Alinhac conjecture on obtaining time-growth properties or
deriving uniform bounds for the top-order energy for the wave-type equations. 
On one hand, in \cite{Deng2020}, the authors studied the large-time behavior of the solutions to a class of three-dimensional nonlinear wave equations satisfying the weak null condition, and proved that the decay of solutions at high-order
 is strictly slower than $t^{-1}$; this means that the high-order energy of the solutions cannot be uniformly bounded, and this establishes Alinhac conjecture concerning the blowup-at-infinity for this problem.
On the other hand, there exists many progress on establishing uniform bounds for the top-order energy of solutions for various classes of wave-type equations.
In the recent work \cite{Cai} on a class of two-dimensional quasilinear wave equations with hidden strong null structure, as well as in \cite{Cai2020} on two-dimensional incompressible elastodynamics, the authors obtained a uniform bound for
 the top-order energy. In their proof, the main ingredients include the use of the ghost weight method \cite{Alinhac1} and
  the inherent strong null structure  \cite{Lei16}.
In the work \cite{Wang2014} and \cite{Lei-Wang}, the authors derived a uniform bound for
 the top-order energy for three-dimensional quasilinear null wave equations and 
 for the system of three-dimensional incompressible isotropic elastodynamics, respectively; their 
 proof was based on a novel use of the null condition at the top-order energy level. 
Furthermore, as mentioned earlier when introducing the hyperboloid foliation method, similar results 
were discovered for three-dimensional coupled wave-Klein-Gordon equations \cite{PLF-YM-book} and for evolving relativistic membranes \cite{PLF-CW}.


\paragraph{Main challenge.}

Let us describe here the main challenge in deriving a uniform bound for the top-order energy of the solutions to \eqref{eq:model-2d}, which are mainly due to the nature of the (slow!) decay of two-dimensional waves. Recall that the ghost energy estimate for the equation $-\Box u = f$
reads 
\be
\int_{\RR^2} |\del u(t)|^2 \, dx
+
\sum_a \int_{t_0}^t \int_{\RR^2}  {|G_a u|^2 \over \langle t-|x| \rangle^{1+\delta}}  \, dxdt
\lesssim
\int_{\RR^2} |\del u(t_0)|^2 \, dx
+
\int_{t_0}^t \int_{\RR^2} \big| f \del_t u \big| \, dxdt,
\ee
in which $\delta > 0$ and $G_a = (x_a / |x|) \del_t + \del_a$ are the so-called ``good derivatives'' associated with
 the ghost weight energy estimate.
Returning to our model \eqref{eq:model-2d}, we see that, when estimating the top-order energy of the solution $w$, we need to control the time-integral 
\be
\int_{t_0}^t \| P^{\gamma \alpha \beta} \del_t \del_\gamma w \del_\alpha \Gamma^I w \del_\beta \Gamma^I w \|_{L^1} \, dt,
\ee
with $|I| = N$ and $\Gamma \in \{ \del_\alpha, L_a, \Omega_{ab}, L_0 \}$. 
(See \eqref{eq:quasi-EE} below for further details.) 
Then we find 
$$
\aligned
& 
\int_{t_0}^t \| P^{\gamma \alpha \beta} \del_t \del_\gamma w \del_\alpha \Gamma^I w \del_\beta \Gamma^I w \|_{L^1} \, dt
\\
&\lesssim
\sum_a \int_{t_0}^t \| \del \del w G_a \Gamma^I w \del \Gamma^I w \|_{L^1} \, dt + \textsl{similar}
\\
&\lesssim
\sum_a \int_{t_0}^t \| \del \del w G_a \Gamma^I w \|_{L^2} \|\del \Gamma^I w  \|_{L^2} \, dt + \textsl{similar}
\\
&\lesssim
\Big( \int_{t_0}^t \| \del \del w \langle t-|x| \rangle^{1/2+\delta/2} \|^2_{L^\infty} \, dt \Big)^{1/2} \Big( \int_{t_0}^t \Big\| {G_a \Gamma^I w \over \langle t-|x| \rangle^{1/2+\delta/2}}  \Big\|^2_{L^2} \, dt \Big)^{1/2} \sup_{\tau \in [t_0, t]} \|\del \Gamma^I w (\tau) \|_{L^2} + \textsl{similar},
\endaligned
$$
where ``similar'' stands for integral terms that will be similarly controled.
It is easy to see that {\sl even if we assume} that $w$ enjoys the properties enjoyed by linear waves, i.e.
\be
\| \del \del w \langle t-|x| \rangle^{1/2+\delta/2} \|^2_{L^\infty} 
\lesssim \langle t \rangle^{-1},
\qquad
\int_{t_0}^t \Big\| {G_a \Gamma^I w \over \langle t-|x| \rangle^{1/2+\delta/2}}  \Big\|^2_{L^2} \, dt + \|\del \Gamma^I w  \|_{L^2}^2
\lesssim 1,
\ee
then 
we are still left with a {\sl bad growth} in time associated with  $(\log \langle t\rangle )^{1/2}$ and the latter
 blows up at infinity.


\paragraph{Our strategy of the proof.}

We now present our new key technique which helps to overcome the aforementioned challenge.
In our approach we use hyperboloids $\Hcal_s = \{ (t, x) : t^2 = s^2 + |x|^2 \}$ in order to foliate the spacetime, and we 
estimate the energy of the wave equation along these hyperboloids, as in~\cite{Klainerman85, Hormander, PLF-YM-book}; 
see Lemma \ref{lem:EE} below. 
We find that in the energy estimate we integrate with respect to the hyperbolic time $s = \sqrt{t^2 - |x|^2}$ instead of $t$, and this allows us to take advantage of the {\sl $t-|x|$ decay},  
which is precisely provided by the ghost weight energy estimate \cite{Alinhac1}. 
A second key ingredient for our proof is the conformal--type energy estimate and its quasilinear version 
 \cite{Wong, YM-HH, YM0, PLF-JO},  leading to almost sharp bounds on the norms 
 $\| \cdot \|_{L^2_f(\Hcal_s)}$ (defined in \eqref{eq:def-norm}, below) of the following terms ($|I| \leq N$):  
$$
(s/t) \Gamma^I w,
\qquad
(s/t)  L_a \Gamma^I w,
\qquad
(s/t) \Omega_{ab} \Gamma^I w,
\qquad
(s/t) L_0 \Gamma^I w. 
$$ 
In addition, we also rely on a version of the estimate for null forms, which is well-adapted to the hyperboloidal foliation setting. This reads as follows~\cite{PLF-YM-book}: 
\be
\aligned
|P^{\gamma \alpha\beta} \del_\gamma w \del_\alpha \del_\beta w|
&\lesssim
(s/t)^2 | \del_t w \del_t\del_t w |
+
\sum_a \big|{L_a w \over t} \del \del w \big|
+
\sum_a \big|\del w  {L_a \del w \over t} \big|
+
t^{-1} | \del w \del w|. 
\endaligned
\ee
We refer to Lemma \ref{lem:null-classical} for further details. Based on this null form estimate, we roughly expect
\be
\big\| P^{\gamma \alpha\beta} \del_\gamma w \del_\alpha \del_\beta w \big\|_{L^2_f(\Hcal_s)}
\lesssim s^{-2},
\ee
which, now, is an integrable function. For a better understanding that null condition can compensate the slow decay of two-dimensional waves, we provide here a comparison with
three-dimensional wave equations with general quadratic nonlinearity. 
Namely, if $u$ denotes a (local) solution to 
$$
-\Box u = \del_t u \del_t \del_t u, \qquad \text{   in  } \RR^{1+3},
$$
then we find that the best decay we can get is
$$
\big\| \del_t u \del_t \del_t u \big\|_{L^2_f(\Hcal_s)}
\lesssim \big\| (t/s) \del_t u \big\|_{L^\infty(\Hcal_s)} \big\| (s/t) \del_t \del_t u \big\|_{L^2_f(\Hcal_s)}
\lesssim s^{-1},
$$
which non-integrable. Clearly, the null condition forces the nonlinearities to decay fast, even if in lower dimension. 
 

\paragraph{Outline of this paper.}

In Section \ref{sec:pre}, we provide some basic properties os wave equations and we introduce the hyperboloidal foliation framework. Next, we derive a quasilinear version of conformal-type energy estimates on hyperboloids; see Section \ref{sec:conformal}. Finally, Section \ref{sec:proof} is devoted to the proof of Theorem \ref{thm:main}.


\section{Fundamental energy estimate}
\label{sec:pre}

\subsection{Notation}

We work in the spacetime $\RR^{2+1}$ with metric $m = \text{diag}(-1, 1, 1)$, and the indices are raised or lowered by the metric $m$. As usual, we use 
$$
\del_\alpha = \del_{x^\alpha}
$$
to denote partial derivatives, with $x^0 = t$ (i.e. $x_0 = -t$). The rotating vector fields are denoted by
$$
\Omega_{ab} = x_a \del_b - x_b \del_a, 
\qquad
a, b \in \{1, 2\}.
$$
We represent the Lorentz boosts by
$$
L_a = x_a \del_t + t \del_a,
\qquad
a \in \{1, 2\},
$$
and the scaling vector field is denoted by
$$
L_0 = x^\alpha \del_\alpha.
$$
We will use $\Gamma$ to denote a vector field in the set $V := \{ \del_0, \del_1, \del_2, L_1, L_2, \Omega_{12}, L_0 \}$.

On the other hand, we use $s \geq s_0 = 2$ to denote the hyperbolic time, and the hyperboloids are denoted by
$$
\Hcal_s = \{ (t, x) : t^2 = |x|^2 + s^2\}.
$$
In order to fit well with the hyperboloidal foliation of the interior of the cone 
$$
\Kcal = \{ (t, x) : t \geq |x| + 1 \},
$$
we first introduce (see \cite{PLF-YM-book}) the hyperboloidal frame
\be 
\overline{\del}_0 = \del_s = (s/t) \del_t,
\qquad
\overline{\del}_a = {L_a \over t} = {x_a \over t} \del_t + \del_a, \quad a = 1, 2.
\ee
The original partial derivatives can be expressed by the hyperboloidal frame, which reads
\be  
\del_\alpha = \Psi^\beta_\alpha \overline{\del}_\beta,
\ee
with the transition matrix 
\be
(\Psi_\alpha^{ \beta})=
\begin{pmatrix}
{t/s} & 0 &   0    \\
-{x_1 / s} & 1  & 0    \\
-{x_2 / s} &  0  &  1   
\end{pmatrix}.
\ee
Next, we introduce (see \cite{PLF-YM-book}) the so-called semi-hyperboloidal frame 
\be 
\underdel_0 = \del_t,
\qquad
\underdel_a = \overline{\del}_a, \quad a = 1, 2.
\ee
The usual partial derivatives can also be expressed by the semi-hyperboloidal frame
$$
\del_t = \underdel_0,
\qquad
\del_a = - {x_a \over t} \underdel_0 + \underdel_a,
$$
and this will be used when estimating the null forms.

Worth to mention, for all points $(t, x) \in \Hcal_s \cap \Kcal$ with $s \geq 2$ we have
$$
|x| \leq t,
\qquad
s \leq t \leq s^2,
\qquad
t \leq t+|x| \leq 2t.
$$

Given a sufficiently nice function $u$ with support $\Kcal$, its weighted norms on hyperboloids $\Hcal_s$ ($s \geq 2$) are defined by
\bel{eq:def-norm}  
\|u\|_{L^p_f(\Hcal_s)}^p
=
\int_{\Hcal_s^*} \big| u (t, x) \big|^p \, dx
:=
\int_{\RR^2} \big| \widetilde{u} (s, x) \big|^p \, dx,
\qquad
\widetilde{u} (s, x) := u \big(\sqrt{s^2 + |x|^2}, x\big),
\ee
for all $1 \leq p < +\infty$, in which $\Hcal_s^*$ represents the projection of $\Hcal_s$ onto the slice $\{(t, x) : t=2 \}$, and the $L^\infty$ norm $\| \cdot \|_{L^\infty(\Hcal_s)}$ is defined in the natural way.

\subsection{Energy estimate}

Given a sufficiently nice function $u$ on $\Hcal_s$, we define its (natural) energy and conformal energy respectivly by (following \cite{PLF-YM-book, YM-HH, YM0})
\be 
\aligned
E(u, s) 
&:= \int_{\Hcal_s^*} \Big( \big(\del_t u \big)^2+ \sum_a \big(\del_a u \big)^2+ 2 (x^a/t) \del_t u \del_a u  \Big) \, dx
\\
&= \int_{\Hcal_s^*} \Big( \big( (s/t)\del_t u \big)^2+ \sum_a \big(\underdel_a u \big)^2 \Big) \, dx
\\
&= \int_{\Hcal_s^*} \Big( \big( \underdel_\perp u \big)^2+ \sum_a \big( (s/t)\del_a u \big)^2+ \sum_{a<b} \big( t^{-1}\Omega_{ab} u \big)^2 \Big) \, dx,
\\
E_{con} (u, s)
&:= \int_{\Hcal_s^*} \Big( \sum_a \big( s \underdel_a u \big)^2 + \big( K u + u \big)^2 \Big) \, dx,
\endaligned
\ee
with $a, b \in\{1, 2\}$, the orthogonal vector field $\underdel_\perp:= L_0/t = \del_t+ (x^a / t)\del_a$ and $K u := \big( s \del_s + 2 x^a \underdel_a \big) u$. Easily, we see that
$$
\int_{\Hcal_s^*} \sum_\alpha \big( (s/t)\del_\alpha u \big)^2 + (\Gamma u/t)^2 \, dx
\lesssim E(u, s),
$$
with $\Gamma \in \{ \del_\alpha, L_a, L_0, \Omega_{ab}\}$.

\begin{lemma}[Energy estimates on hyperboloids]\label{lem:EE}
Consider the wave equation
$$
-\Box u = h, 
\qquad
(u, \del_t u)(s_0, \cdot) = (u_0, u_1).
$$
We have the following three kinds of energy estimates on hyperboloids.
\bei
\item The standard energy estimates (see \cite{PLF-YM-book}):
\bel{eq:E-E1}
E(u, s)^{1/2}
\leq
E(u, s_0)^{1/2}
+
\int_{s_0}^s \| h \|_{L^2_f(\Hcal_{\tau})} \, d\tau.
\ee

\item The conformal energy estimates (see \cite{Wong, YM-HH, YM0}):
\bel{eq:E-E2}
E_{con} (u, s)^{1/2}
\leq
E_{con} (u, s_0)^{1/2}
+
\int_{s_0}^s \tau \| h \|_{L^2_f(\Hcal_{\tau})} \, d\tau.
\ee

\item The $L^2$ norm estimates (see \cite{YM0}):
\bel{eq:E-E3} 
\| (s/t) u \|_{L^2_f(\Hcal_s)}
\leq 
\| u \|_{L^2_f(\Hcal_{s_0})}
+
\int_{s_0}^s {E_{con}(u, \tau)^{1/2} \over \tau} \, d\tau.
\ee

\eei
\end{lemma}

\begin{proof}
We revisit the proof for the conformal energy estimates only.

First recall
$$
\del_t = {t\over s} \del_s,
\qquad
\del_a = -{x_a\over s} \del_s + \overline{\del}_a,
$$
and the wave operator $-\Box$ can thus be expressed in terms of $\overline{\del}_\alpha$, which reads
$$
\aligned
-\Box u
&= \del_t \del_t u - \del_a \del^a u
\\
&={t\over s} \del_s \big( {t\over s} \del_s u \big) - \big(-{x_a\over s} \del_s  + \overline{\del}_a \big)  \big(-{x^a\over s} \del_s u + \overline{\del}^a u \big)
\\
&= {t^2\over s^2} \del_s \del_s u + {t\over s} {1\over s} {s\over t} \del_s u - {t\over s} t {1\over s^2} \del_s u
- {x_a x^a \over s^2} \del_s \del_s u + {x_a x^a \over s^2} {1\over s^2} \del_s u 
\\
&+ {x_a \over s} \del_s \overline{\del}^a u
+ {x_a \over s} \overline{\del}^a \del_s u + {2\over s} \del_s u - \overline{\del}_a \overline{\del}^a u
\\
&= \del_s \del_s u + 2 {x^a \over s} \overline{\del}_a \del_s u + {2\over s} \del_s u -  \overline{\del}_a \overline{\del}^a u
\\
&= s^{-1} \del_s \big( s \del_s u + 2 x^a \overline{\del}_a u + u \big) -  \overline{\del}_a \overline{\del}^a u.
\endaligned
$$

In succession, we have
$$
\aligned
&s \big( s\del_s u + 2x^a \overline{\del}_a u + u \big) \big( -\Box u \big)
=
&{1\over 2} \del_s \big( s\del_s u + 2x^a \overline{\del}_a u + u \big)^2
- s \big( s\del_s u + 2x^a \overline{\del}_a u + u \big) \overline{\del}_b \overline{\del}^b u.
\endaligned
$$
By some simple calculations
$$
\aligned
s^2 \del_s u \overline{\del}_b \overline{\del}^b u
&=
s^2 \overline{\del}_b \big( \del_s u \overline{\del}^b u \big) - {1\over 2} s^2 \del_s \big( \overline{\del}_b u \overline{\del}^b u  \big)
\\
&=
s^2 \overline{\del}_b \big( \del_s u \overline{\del}^b u \big) 
-  {1\over 2} \del_s \big( s^2 \overline{\del}_b u \overline{\del}^b u  \big) + s \overline{\del}_b u \overline{\del}^b u,
\endaligned
$$
$$
\aligned
2 s x^a \overline{\del}_a u \overline{\del}_b \overline{\del}^b u
&= 2s x^a \overline{\del}_b \big( \overline{\del}_a u \overline{\del}^b u \big)
- s x^a \overline{\del}_a \big( \overline{\del}_b u \overline{\del}^b u \big)
\\
&=
2 s \overline{\del}_b \big( x^a  \overline{\del}_a u \overline{\del}^b u \big)
- 2 s \overline{\del}_a u \overline{\del}^a u
- s  \overline{\del}_a \big( x^a \overline{\del}_b u \overline{\del}^b u \big)
+ 2 s \overline{\del}_a u \overline{\del}^a u
\\
&=
2 s \overline{\del}_b \big( x^a  \overline{\del}_a u \overline{\del}^b u \big)
- s  \overline{\del}_a \big( x^a \overline{\del}_b u \overline{\del}^b u \big)
\endaligned
$$
and
$$
\aligned
s u \overline{\del}_b \overline{\del}^b u
=
s \overline{\del}_b \big( u \overline{\del}^b u \big) - s \overline{\del}_b u \overline{\del}^b u,
\endaligned
$$
we arrive at
$$
\aligned
& s \big( s\del_s u + 2x^a \overline{\del}_a u + u \big) \big( -\Box u \big)
= s \big( K u + u  \big) h
\\
=
& {1\over 2} \del_s \big( s\del_s u + 2x^a \overline{\del}_a u + u \big)^2
+
{1\over 2} \del_s \big( s^2 \overline{\del}_b u \overline{\del}^b u  \big) 
-
s^2 \overline{\del}_b \big( \del_s u \overline{\del}^b u \big) 
-
2 s \overline{\del}_b \big( x^a  \overline{\del}_a u \overline{\del}^b u \big)
\\
+& s  \overline{\del}_a \big( x^a \overline{\del}_b u \overline{\del}^b u \big)
- s \overline{\del}_b \big( u \overline{\del}^b u \big).
\endaligned
$$
Integrating the above identity yields the desired energy estimates.

\end{proof}

The following lemma from \cite{Dong1912} helps bound $L^2$--type norm for $L_0 u$.

\begin{lemma}\label{lem:L2type}
Let $u$ be supported in $\Kcal$, then it holds that
\bel{eq:E-E4} 
\aligned
&\big\| (s / t) L_0 u \big\|_{L^2_f(\Hcal_s)}
+
\sum_a \big\| (s / t) L_a u \big\|_{L^2_f(\Hcal_s)}
+
\big\| (s / t) \Omega_{12} u \big\|_{L^2_f(\Hcal_s)}
\\
\lesssim
&\big\| (s / t) u \big\|_{L^2_f(\Hcal_s)}
+
E_{con} (u, s)^{1/2}.
\endaligned
\ee

\end{lemma}

\begin{proof}

We revisit the proof for \eqref{eq:E-E4} in \cite{Dong1912}.

First, the definition of the conformal energy easily deduces that
$$
\sum_a \big\| (s / t) L_a u \big\|_{L^2_f(\Hcal_s)}
\leq
E_{con} (u, s)^{1/2}.
$$
Then the relation $\Omega_{12} = t^{-1} \big(x_1 L_2 - x_2 L_1\big)$ implies
$$
\big\| (s/t) \Omega_{12} u \big\|_{L^2_f(\Hcal_s)}
\lesssim
\sum_a \big\| (s / t) L_a u \big\|_{L^2_f(\Hcal_s)}
\leq
E_{con} (u, s)^{1/2}.
$$

Next, we consider the estimate for the scaling vector field $L_0$, and we write the scaling vector field in the hyperboloidal frame
$$
L_0
=
t \del_t + x^a \del_a
=
s \del_s + x^a \underdel_a.
$$
Note that
$$
K u 
= \big( s \del_s + 2 x^a \underdel_a \big) u
= L_0 u + x^a \underdel_a u. 
$$
By the triangle inequality, we find
$$
\aligned
\big\| (s/t) L_0 u \big\|_{L^2_f(\Hcal_s)}
&\leq
\big\| (s/t) (K +1) u \big\|_{L^2_f(\Hcal_s)}
+
\big\| (s/t)  x^a \underdel_a u \big\|_{L^2_f(\Hcal_s)}
+
\big\| (s/t) u \big\|_{L^2_f(\Hcal_s)}
\\
&=
\big\| (s/t) (K+1) u \big\|_{L^2_f(\Hcal_s)}
+
\big\| (s/t)  (x^a/t) L_a u \big\|_{L^2_f(\Hcal_s)}
+
\big\| (s/t) u \big\|_{L^2_f(\Hcal_s)}.
\endaligned
$$
Finally we recall that within the cone $\Kcal$ it holds that $s<t, |x^a| < t$, and we thus obtain
$$
\big\| (s/t) L_0 u \big\|_{L^2_f(\Hcal_s)}
\lesssim
\big\| (s/t) u \big\|_{L^2_f(\Hcal_s)} + E_{con}(u, s)^{1/2}.
$$
The proof is now complete.
\end{proof}


\section{Conformal energy for quasilinear wave operators}\label{sec:conformal}

For the quasilinear wave equation (later on, we will take $u = \Gamma^I w$ with $|I| = N$)
$$
g^{\alpha \beta} \del_\alpha \del_\beta u = h,
\qquad
g^{\alpha \beta} = -m^{\alpha \beta} + P^{\gamma \alpha\beta} \del_\gamma w,
$$
we want to show the following conformal type energy estimates in analogues to \eqref{eq:E-E2}, which have been proved in \cite{YM-HH, YM0}.

\begin{proposition}\label{prop:quasi-conformal}
Consider the quasilinear wave equation
\be 
g^{\alpha \beta} \del_\alpha \del_\beta u = h,
\qquad
g^{\alpha \beta} = -m^{\alpha \beta} + P^{\gamma \alpha\beta} \del_\gamma w,
\ee
then we have
\bel{eq:quasi-conformal}
\widetilde{E}_{con} (u, s)
=
\widetilde{E}_{con} (u, s_0)
+
\int_{s_0}^s \int_{\Hcal_{\tau}^*} \tau Z h - \tau Z M_1 - M_4 \, dx d\tau.
\ee
In the above, we used the notations
\be 
\aligned
&
\widetilde{E}_{con} (u, s)
=
\int_{\Hcal_{s}^*} {1\over 2} Z^2 + M_2  \, dx,
\\
&
Z 
= 
g^{\alpha\beta} \Psi^0_\alpha \Psi^0_\beta s \del_s u 
+ 2 g^{\alpha\beta} \Psi^0_\alpha \Psi^a_\beta s \overline{\del}_a u 
+ g^{\alpha\beta} \Psi^\gamma_\alpha (\overline{\del}_\gamma \Psi^0_\beta) s u
- g^{\alpha\beta} \Psi^ 0_\alpha \Psi^0_\beta u,
\\
&M_1
=
 {1\over s} \del_s \big( g^{\alpha\beta} \Psi^ 0_\alpha \Psi^0_\beta \big)  u   
- \del_s (g^{\alpha\beta} \Psi^0_\alpha \Psi^0_\beta) \del_s u 
-2 {1\over s} g^{\alpha\beta} \Psi^0_\alpha \Psi^a_\beta \overline{\del}_a u
\\
&\quad- 2 \del_s (g^{\alpha\beta} \Psi^0_\alpha \Psi^a_\beta ) \overline{\del}_a u
-{1\over s} g^{\alpha\beta} \Psi^\gamma_\alpha (\overline{\del}_\gamma \Psi^0_\beta) u
- \del_s \big( g^{\alpha\beta} \Psi^\gamma_\alpha  (\overline{\del}_\gamma \Psi^0_\beta) \big) u,
\\
&
M_2 
= - {1\over 2}  s^2 g^{\alpha\beta} \Psi^0_\alpha \Psi^0_\beta  g^{a b} \overline{\del}_a u \overline{\del}_b u,
\\
&M_4
=
{1\over 2} \del_s \big( s^2 g^{\alpha\beta} \Psi^0_\alpha \Psi^0_\beta  g^{a b} \big) \overline{\del}_a u \overline{\del}_b u
- s^2 \overline{\del}_a \big( g^{\alpha\beta} \Psi^0_\alpha \Psi^0_\beta  g^{a b}  \big)  \del_s u \overline{\del}_b u
\\
&\quad
- 2 \overline{\del}_{a'} \big( s^2 g^{\alpha\beta} \Psi^0_\alpha \Psi^a_\beta g^{a' b'} \big) \overline{\del}_a u \overline{\del}_{b'} u
+ \overline{\del}_a \big( s^2 g^{\alpha\beta} \Psi^0_\alpha \Psi^a_\beta g^{a' b'} \big) \overline{\del}_{a'} u \overline{\del}_{b'} u,
\\
&\quad
- \overline{\del}_a \big( s^2 g^{\alpha\beta} \Psi^\gamma_\alpha (\overline{\del}_\gamma \Psi^0_\beta) g^{a b} \big) u \overline{\del}_b u 
- s^2 g^{\alpha\beta} \Psi^\gamma_\alpha (\overline{\del}_\gamma \Psi^0_\beta) g^{a b} \overline{\del}_a u \overline{\del}_b u
\\
&\quad
+ \overline{\del}_a \big(s g^{\alpha \beta} \Psi^0_\alpha \Psi^0_\beta  g^{ab} \big) u \overline{\del}_b u 
+ s g^{\alpha \beta} \Psi^0_\alpha \Psi^0_\beta  g^{ab} \overline{\del}_a u \overline{\del}_b u.
\endaligned
\ee

\end{proposition}

\begin{proof}

Inserting the relation
$$
\del_\alpha = \Psi^\beta_\alpha \overline{\del}_\beta
$$
into the above equation, we have
$$
\aligned
g^{\alpha\beta} \del_\alpha \del_\beta u
&=
g^{\alpha \beta} (\Psi^\gamma_\alpha \overline{\del}_\gamma) (\Psi^\eta_\beta \overline{\del}_\eta u)
\\
&=
g^{\alpha \beta} \Psi^\gamma_\alpha \Psi^\eta_\beta \overline{\del}_\gamma \overline{\del}_\eta u
+
g^{\alpha\beta} \Psi^\gamma_\alpha (\overline{\del}_\gamma \Psi^\eta_\beta) \overline{\del}_\eta u.
\endaligned
$$
To succeed, we get
$$
\aligned
g^{\alpha\beta} \del_\alpha \del_\beta u
&=
g^{\alpha\beta} \Psi^0_\alpha \Psi^0_\beta \del_s\del_s u 
+ 2 g^{\alpha\beta} \Psi^0_\alpha \Psi^a_\beta \del_s \overline{\del}_a u
+ g^{\alpha \beta} \Psi^a_\alpha \Psi^b_\beta \overline{\del}_a \overline{\del}_b u
\\
&+g^{\alpha \beta} \Psi^\gamma_\alpha (\overline{\del}_\gamma \Psi^0_\beta) \del_s u
+ g^{\alpha\beta} \Psi^\gamma_\alpha (\overline{\del}_\gamma \Psi^a_\beta) \overline{\del}_a u
\\
&=
\del_s \big( g^{\alpha\beta} \Psi^0_\alpha \Psi^0_\beta \del_s u \big)
-\del_s \big( g^{\alpha\beta} \Psi^0_\alpha \Psi^0_\beta \big) \del_s u
+
2 \del_s \big( g^{\alpha\beta} \Psi^0_\alpha \Psi^a_\beta \overline{\del}_a u \big)
-2 \del_s \big( g^{\alpha \beta} \Psi^0_\alpha \Psi^a_\beta \big) \overline{\del}_a u
\\
&+ g^{\alpha \beta} \Psi^a_\alpha \Psi^b_\beta \overline{\del}_a \overline{\del}_b u
+ \del_s \big( g^{\alpha\beta} \Psi^\gamma_\alpha (\overline{\del}_\gamma \Psi^0_\beta) u \big)
- \del_s \big( g^{\alpha\beta} \Psi^\gamma_\alpha (\overline{\del}_\gamma \Psi^0_\beta) \big) u,
\endaligned
$$
in which we used the relation $\overline{\del}_\gamma \Psi^a_\beta = 0$.

Furthermore, we find
$$
\aligned
&  g^{\alpha\beta} \del_\alpha \del_\beta u
\\
&=
{1\over s} \del_s \big( g^{\alpha\beta} \Psi^0_\alpha \Psi^0_\beta s \del_s u \big)
- {1\over s}  g^{\alpha\beta} \Psi^0_\alpha \Psi^0_\beta \del_s u 
- \del_s \big( g^{\alpha\beta} \Psi^0_\alpha \Psi^0_\beta \big) \del_s u 
+ 2 {1\over s} \del_s \big( g^{\alpha\beta} \Psi^0_\alpha \Psi^a_\beta s \overline{\del}_a u \big)
\\
&- 2 {1\over s} g^{\alpha\beta} \Psi^0_\alpha \Psi^a_\beta \overline{\del}_a u 
- 2 \del_s \big( g^{\alpha\beta} \Psi^0_\alpha \Psi^a_\beta \big) \overline{\del}_a u 
+ {1\over s} \del_s \big( g^{\alpha\beta} \Psi^\gamma_\alpha (\overline{\del}_\gamma \Psi^0_\beta) s u \big)
\\
&- {1\over s}  g^{\alpha\beta} \Psi^\gamma_\alpha (\overline{\del}_\gamma \Psi^0_\beta) u 
- \del_s \big( g^{\alpha\beta} \Psi^\gamma_\alpha (\overline{\del}_\gamma \Psi^0_\beta) \big) u 
+g^{\alpha\beta} \Psi^a_\alpha \Psi^b_\beta \overline{\del}_a \overline{\del}_b u
\\
&=
{1\over s} \del_s \Big( g^{\alpha\beta} \Psi^0_\alpha \Psi^0_\beta s \del_s u 
+ 2 g^{\alpha\beta} \Psi^0_\alpha \Psi^a_\beta s \overline{\del}_a u 
+ g^{\alpha\beta} \Psi^\gamma_\alpha (\overline{\del}_\gamma \Psi^0_\beta) s u
- g^{\alpha\beta} \Psi^ 0_\alpha \Psi^0_\beta u \Big)
\\
& + {1\over s} \del_s \big( g^{\alpha\beta} \Psi^ 0_\alpha \Psi^0_\beta \big)  u                  
- \del_s (g^{\alpha\beta} \Psi^0_\alpha \Psi^0_\beta) \del_s u       
-2 {1\over s} g^{\alpha\beta} \Psi^0_\alpha \Psi^a_\beta \overline{\del}_a u
\\
&- 2 \del_s (g^{\alpha\beta} \Psi^0_\alpha \Psi^a_\beta ) \overline{\del}_a u
-{1\over s} g^{\alpha\beta} \Psi^\gamma_\alpha (\overline{\del}_\gamma \Psi^0_\beta) u
- \del_s \big( g^{\alpha\beta} \Psi^\gamma_\alpha  (\overline{\del}_\gamma \Psi^0_\beta) \big) u
+g^{a b} \overline{\del}_a \overline{\del}_b u,
\endaligned
$$
in which we used the relation $g^{a b} = g^{\alpha\beta} \Psi^a_\alpha \Psi^b_\beta$.

Next, we denote 
$$
\aligned
Z 
&= 
g^{\alpha\beta} \Psi^0_\alpha \Psi^0_\beta s \del_s u 
+ 2 g^{\alpha\beta} \Psi^0_\alpha \Psi^a_\beta s \overline{\del}_a u 
+ g^{\alpha\beta} \Psi^\gamma_\alpha (\overline{\del}_\gamma \Psi^0_\beta) s u
- g^{\alpha\beta} \Psi^ 0_\alpha \Psi^0_\beta u,
\\
M_1
&=
 {1\over s} \del_s \big( g^{\alpha\beta} \Psi^ 0_\alpha \Psi^0_\beta \big)  u   
- \del_s (g^{\alpha\beta} \Psi^0_\alpha \Psi^0_\beta) \del_s u 
-2 {1\over s} g^{\alpha\beta} \Psi^0_\alpha \Psi^a_\beta \overline{\del}_a u
- 2 \del_s (g^{\alpha\beta} \Psi^0_\alpha \Psi^a_\beta ) \overline{\del}_a u
\\
&-{1\over s} g^{\alpha\beta} \Psi^\gamma_\alpha (\overline{\del}_\gamma \Psi^0_\beta) u
- \del_s \big( g^{\alpha\beta} \Psi^\gamma_\alpha  (\overline{\del}_\gamma \Psi^0_\beta) \big) u,
\endaligned
$$
then we are led to
$$
\aligned
s Z g^{\alpha\beta} \del_\alpha \del_\beta u
=
{1\over 2} \del_s \big( Z^2 \big)
+ s Z M_1
+ s Z g^{a b} \overline{\del}_a \overline{\del}_b u.
\endaligned
$$

As for the last term, we split it into four parts:
$$
\aligned
A_1
&= s g^{\alpha\beta} \Psi^0_\alpha \Psi^0_\beta s \del_s u  g^{a b} \overline{\del}_a \overline{\del}_b u
\\
&= s^2 \overline{\del}_a \big( g^{\alpha\beta} \Psi^0_\alpha \Psi^0_\beta \del_s u  g^{a b} \overline{\del}_b u \big)
- s^2 \overline{\del}_a \big( g^{\alpha\beta} \Psi^0_\alpha \Psi^0_\beta  g^{a b}  \big)  \del_s u \overline{\del}_b u
- {1\over 2} s^2 g^{\alpha\beta} \Psi^0_\alpha \Psi^0_\beta  g^{a b} \del_s \big( \overline{\del}_a u \overline{\del}_b u \big)
\\
&= s^2 \overline{\del}_a \big( g^{\alpha\beta} \Psi^0_\alpha \Psi^0_\beta \del_s u  g^{a b} \overline{\del}_b u \big)
- {1\over 2} \del_s \big( s^2 g^{\alpha\beta} \Psi^0_\alpha \Psi^0_\beta  g^{a b} \overline{\del}_a u \overline{\del}_b u \big)
+ {1\over 2} \del_s \big( s^2 g^{\alpha\beta} \Psi^0_\alpha \Psi^0_\beta  g^{a b} \big) \overline{\del}_a u \overline{\del}_b u
\\
&- s^2 \overline{\del}_a \big( g^{\alpha\beta} \Psi^0_\alpha \Psi^0_\beta  g^{a b}  \big)  \del_s u \overline{\del}_b u,
\endaligned
$$

$$
\aligned
A_2
&= 2 s g^{\alpha\beta} \Psi^0_\alpha \Psi^a_\beta s \overline{\del}_a u   g^{a' b'} \overline{\del}_{a'} \overline{\del}_{b'} u
\\
&= 2 s^2 g^{\alpha\beta} \Psi^0_\alpha \Psi^a_\beta g^{a' b'} \overline{\del}_{a'} \big( \overline{\del}_a u \overline{\del}_{b'} u \big)
- s^2 g^{\alpha\beta} \Psi^0_\alpha \Psi^a_\beta g^{a' b' } \overline{\del}_{a} \big( \overline{\del}_{a'} u \overline{\del}_{b'} u \big)
\\
&= \overline{\del}_{a'} \big( 2 s^2 g^{\alpha\beta} \Psi^0_\alpha \Psi^a_\beta g^{a' b'} \overline{\del}_a u \overline{\del}_{b'} u \big)
- 2 \overline{\del}_{a'} \big( s^2 g^{\alpha\beta} \Psi^0_\alpha \Psi^a_\beta g^{a' b'} \big) \overline{\del}_a u \overline{\del}_{b'} u
\\
&- \overline{\del}_a \big( s^2 g^{\alpha\beta} \Psi^0_\alpha \Psi^a_\beta g^{a' b'} \overline{\del}_{a'} u \overline{\del}_{b'} u \big)
+ \overline{\del}_a \big( s^2 g^{\alpha\beta} \Psi^0_\alpha \Psi^a_\beta g^{a' b'} \big) \overline{\del}_{a'} u \overline{\del}_{b'} u,
\endaligned
$$
and
$$
\aligned
A_3
&= s g^{\alpha\beta} \Psi^\gamma_\alpha (\overline{\del}_\gamma \Psi^0_\beta) s u g^{a b} \overline{\del}_a \overline{\del}_b u
\\
&= s^2 g^{\alpha\beta} \Psi^\gamma_\alpha (\overline{\del}_\gamma \Psi^0_\beta) g^{a b} \overline{\del}_a \big( u \overline{\del}_b u \big)
- s^2 g^{\alpha\beta} \Psi^\gamma_\alpha (\overline{\del}_\gamma \Psi^0_\beta) g^{a b} \overline{\del}_a u \overline{\del}_b u
\\
&= \overline{\del}_a \big( s^2 g^{\alpha\beta} \Psi^\gamma_\alpha (\overline{\del}_\gamma \Psi^0_\beta) g^{a b} u \overline{\del}_b u \big)
- \overline{\del}_a \big( s^2 g^{\alpha\beta} \Psi^\gamma_\alpha (\overline{\del}_\gamma \Psi^0_\beta) g^{a b} \big) u \overline{\del}_b u 
- s^2 g^{\alpha\beta} \Psi^\gamma_\alpha (\overline{\del}_\gamma \Psi^0_\beta) g^{a b} \overline{\del}_a u \overline{\del}_b u,
\endaligned
$$

$$
\aligned
A_4
&= -s g^{\alpha \beta} \Psi^0_\alpha \Psi^0_\beta u g^{ab} \overline{\del}_a \overline{\del}_b u
\\
&= -s g^{\alpha \beta} \Psi^0_\alpha \Psi^0_\beta  g^{ab} \overline{\del}_a \big( u \overline{\del}_b u \big)
+ s g^{\alpha \beta} \Psi^0_\alpha \Psi^0_\beta  g^{ab} \overline{\del}_a u \overline{\del}_b u
\\
&=
\overline{\del}_a \big(-s g^{\alpha \beta} \Psi^0_\alpha \Psi^0_\beta  g^{ab}  u \overline{\del}_b u \big)
+ \overline{\del}_a \big(s g^{\alpha \beta} \Psi^0_\alpha \Psi^0_\beta  g^{ab} \big) u \overline{\del}_b u 
+ s g^{\alpha \beta} \Psi^0_\alpha \Psi^0_\beta  g^{ab} \overline{\del}_a u \overline{\del}_b u.
\endaligned
$$

Then, we get
$$
\aligned
s Z g^{a b} \overline{\del}_a \overline{\del}_b u
&=
A_1 + A_2 + A_3 + A_4
=: \del_s M_2 + M_3 + M_4
\\
M_2 
&= - {1\over 2}  s^2 g^{\alpha\beta} \Psi^0_\alpha \Psi^0_\beta  g^{a b} \overline{\del}_a u \overline{\del}_b u,
\\
M_3
&=
s^2 \overline{\del}_a \big( g^{\alpha\beta} \Psi^0_\alpha \Psi^0_\beta \del_s u  g^{a b} \overline{\del}_b u \big)
+ \overline{\del}_{a'} \big( 2 s^2 g^{\alpha\beta} \Psi^0_\alpha \Psi^a_\beta g^{a' b'} \overline{\del}_a u \overline{\del}_{b'} u \big)
\\
&- \overline{\del}_a \big( s^2 g^{\alpha\beta} \Psi^0_\alpha \Psi^a_\beta g^{a' b'} \overline{\del}_{a'} u \overline{\del}_{b'} u \big)
+ \overline{\del}_a \big( s^2 g^{\alpha\beta} \Psi^\gamma_\alpha (\overline{\del}_\gamma \Psi^0_\beta) g^{a b} u \overline{\del}_b u \big)
- \overline{\del}_a \big(s g^{\alpha \beta} \Psi^0_\alpha \Psi^0_\beta  g^{ab}  u \overline{\del}_b u \big),
\\
M_4
&=
{1\over 2} \del_s \big( s^2 g^{\alpha\beta} \Psi^0_\alpha \Psi^0_\beta  g^{a b} \big) \overline{\del}_a u \overline{\del}_b u
- s^2 \overline{\del}_a \big( g^{\alpha\beta} \Psi^0_\alpha \Psi^0_\beta  g^{a b}  \big)  \del_s u \overline{\del}_b u
\\
&- 2 \overline{\del}_{a'} \big( s^2 g^{\alpha\beta} \Psi^0_\alpha \Psi^a_\beta g^{a' b'} \big) \overline{\del}_a u \overline{\del}_{b'} u
+ \overline{\del}_a \big( s^2 g^{\alpha\beta} \Psi^0_\alpha \Psi^a_\beta g^{a' b'} \big) \overline{\del}_{a'} u \overline{\del}_{b'} u,
\\
&- \overline{\del}_a \big( s^2 g^{\alpha\beta} \Psi^\gamma_\alpha (\overline{\del}_\gamma \Psi^0_\beta) g^{a b} \big) u \overline{\del}_b u 
- s^2 g^{\alpha\beta} \Psi^\gamma_\alpha (\overline{\del}_\gamma \Psi^0_\beta) g^{a b} \overline{\del}_a u \overline{\del}_b u
\\
&+ \overline{\del}_a \big(s g^{\alpha \beta} \Psi^0_\alpha \Psi^0_\beta  g^{ab} \big) u \overline{\del}_b u 
+ s g^{\alpha \beta} \Psi^0_\alpha \Psi^0_\beta  g^{ab} \overline{\del}_a u \overline{\del}_b u,
\endaligned
$$
and we find that the term $M_2$ will contribute to the left hand side, $M_3$ will be gone after the spacial integration, and $M_4$ will be moved to the right hand side as a source term.

Now, we have
$$
\aligned
& s Z g^{\alpha\beta} \del_\alpha \del_\beta u
= s Z h
\\
&=
{1\over 2} \del_s \big( Z^2 \big)
+ s Z M_1
+ \del_s M_2 + M_3 + M_4,
\endaligned
$$
and we further get
$$
\int_{\Hcal_{s}^*} {1\over 2} Z^2 + M_2  \, dx
=
\int_{\Hcal_{s_0}^*} {1\over 2} Z^2 + M_2  \, dx
+
\int_{s_0}^s \int_{\Hcal_{\tau}^*} \tau Z h - \tau Z M_1 - M_4 \, dx d\tau. \qedhere
$$
\end{proof}


\subsection{Null form estimates and commutator estimates}

We need the following result to estimate the classical null forms, which can be found in \cite{PLF-YM-book} for instance.

\begin{lemma}\label{lem:null-classical}
Let $P^{\gamma \alpha \beta}$ satisfy the null condition, then for all nice functions $u, v, w$ supported in $\{ (t, x) : t \geq |x| + 1 \}$, it holds 
\be
\aligned 
|P^{\gamma \alpha\beta} \del_\gamma v \del_\alpha \del_\beta u|
&\lesssim
(s/t)^2 | \del_t v \del_t\del_t u |
+
\sum_a |\underdel_a v \del \del u |
+
\sum_a |\del v  \underdel_a \del u|
+
t^{-1} | \del v \del u|,
\\
|P^{\gamma \alpha\beta} \del_\gamma v \del_\alpha u \del_\beta w|
&\lesssim
(s/t)^2 | \del_t v \del_t u \del_t w |
+
\sum_a |\underdel_a v \del_t u \del_t w |
+
\sum_a |\del_t v \underdel_a u \del_t w|
+
\sum_a |\del_t v \del_t u \underdel_a w|.
\endaligned
\ee
\end{lemma}

\begin{proof}
The proof can be found in \cite{PLF-YM-book}, but here we also provide one elementary proof.

We express the term $P^{\gamma \alpha\beta} \del_\gamma v \del_\alpha \del_\beta u$ in the semi-hyperboloidal frame
$\big(\del_t, \underdel_a \big),$
with the relation
$$
\del_a = - {x_a \over t} \del_t + \underdel_a.
$$
We find that
$$
\aligned
P^{\gamma \alpha\beta} \del_\gamma v \del_\alpha \del_\beta u
&=
P^{000} \del_t v \del_t \del_t u + P^{00a} \del_t v \del_t \del_a u + P^{0a0} \del_t v \del_a \del_t u
+ P^{0ab} \del_t v \del_a \del_b u + P^{a00} \del_a v \del_t \del_t u 
\\
&+ P^{a0b} \del_a v \del_t \del_b u
+ P^{ab0} \del_a v \del_b \del_t u + P^{abc} \del_a v \del_b \del_c u
\\
&=
P^{000} \del_t v \del_t \del_t u 
+ P^{00a} \del_t v \big( -{x_a \over t} \del_t \del_t  u + \underdel_a \del_t u \big)
+ P^{0a0} \del_t v \big( -{x_a \over t} \del_t \del_t  u + \underdel_a \del_t u \big)
\\
&+ P^{0ab} \del_t v \big( - {x_a \over t} \del_t  + \underdel_a \big) \big(- {x_b \over t} \del_t u + \underdel_b u \big)
+ P^{a00} \big( - {x_a \over t} \del_t v + \underdel_a v \big) \del_t \del_t u 
\\
&+ P^{a0b} \big( - {x_a \over t} \del_t v + \underdel_a v \big) \big( -{x_b \over t} \del_t \del_t  u + \underdel_b \del_t u \big)
\\
&+ P^{ab0} \big( - {x_a \over t} \del_t v + \underdel_a v \big) \big( -{x_b \over t} \del_t \del_t  u + \underdel_b \del_t u \big)
\\
&+ P^{abc} \big( - {x_a \over t} \del_t v + \underdel_a v \big) \big( - {x_b \over t} \del_t + \underdel_b \big) \big(- {x_c \over t} \del_t u + \underdel_c u \big)
= B_1 + B_2,
\endaligned
$$
in which 
$$
\aligned
B_1 
&=
\del_t v \del_t \del_t u \Big( P^{000} + P^{00a} \big( - {x_a \over t} \big) + P^{0a0} \big( - {x_a \over t} \big)
+ P^{0ab} \big( - {x_a \over t} \big) \big( - {x_b \over t} \big) + P^{a00} \big( - {x_a \over t} \big) 
\\
&+ P^{a0b} \big( - {x_a \over t} \big) \big( - {x_b \over t} \big)
+ P^{ab0} \big( - {x_a \over t} \big) \big( - {x_b \over t} \big) + P^{abc} \big( - {x_a \over t} \big) \big( - {x_b \over t} \big) \big( - {x_c \over t} \big)
\Big),
\\
B_2
&=
P^{00a} \del_t v \underdel_a \del_t u
+ P^{0a0} \del_t v \underdel_a \del_t u
+ P^{0ab} \del_t v \big( -{x_a x_b \over t^3} \del_t u - {x_a \over t} \del_t \underdel_b u - {x_b \over t} \underdel_a \del_t u - {\delta_{ab} \over t} \del_t u 
\\
&+ {x_a x_b \over t^3} \del_t u  + \underdel_a \underdel_b u \big)
+ P^{a00} \underdel_a v \del_t \del_t u
+ P^{a0b} \big(  - {x_a \over t}  \del_t v \underdel_b \del_t u - {x_b \over t} \underdel_a v \del_t \del_t u + \underdel_a v \underdel_b \del_t u   \big)
\\
&+ P^{ab0} \big(  - {x_a \over t}  \del_t v \underdel_b \del_t u - {x_b \over t} \underdel_a v  \del_t \del_t u + \underdel_a v \underdel_b \del_t u   \big)
+ P^{abc} \big( {x_a x_b \over t^2} \del_t v \del_t \underdel_c u + {x_a \delta_{bc} \over t^2} \del_t v \del_t u 
\\
&- {x_a x_b x_c \over t^4} \del_t v \del_t u + {x_a x_c \over t^2} \del_t v \underdel_b \del_t u - {x_a \over t} \del_t v \underdel_b \underdel_c u
- {x_b x_c \over t^3} \underdel_a v \del_t u - {x_b \over t} \underdel_a v \del_t \underdel_c u - {\delta_{bc} \over t} \underdel_a v \del_t u 
\\
&+ {x_b x_c \over t^3} \underdel_a v \del_t u - {x_c \over t} \underdel_a v \underdel_b \del_t u + \underdel_a v \underdel_b \underdel_c u \big).  
\endaligned
$$

Since $P^{\gamma \alpha\beta}$ is null, we have
$$
\aligned
& 
P^{000} {r^3\over t^3} + P^{00a} \big( - {x_a \over t} \big) {r^2\over t^2} + P^{0a0} \big( - {x_a \over t} \big) {r^2\over t^2} 
+ P^{0ab} \big( - {x_a \over t} \big) \big( - {x_b \over t} \big) {r\over t}  + P^{a00} \big( - {x_a \over t} \big) {r^2\over t^2} 
\\
&+ P^{a0b} \big( - {x_a \over t} \big) \big( - {x_b \over t} \big) {r\over t} 
+ P^{ab0} \big( - {x_a \over t} \big) \big( - {x_b \over t} \big) {r\over t}  + P^{abc} \big( - {x_a \over t} \big) \big( - {x_b \over t} \big) \big( - {x_c \over t} \big)
= 0.
\endaligned
$$
This deduces that
$$
\aligned
B_1 
&=
\del_t v \del_t \del_t u \Big( P^{000} \big( 1-{r^3\over t^3} \big) + P^{00a} \big( - {x_a \over t} \big) \big( 1-{r^2\over t^2} \big) + P^{0a0} \big( - {x_a \over t} \big) \big( 1-{r^2\over t^2} \big)
\\
&+ P^{0ab} \big( - {x_a \over t} \big) \big( - {x_b \over t} \big) \big( 1-{r\over t} \big) + P^{a00} \big( - {x_a \over t} \big) \big( 1-{r^2\over t^2} \big)
+ P^{a0b} \big( - {x_a \over t} \big) \big( - {x_b \over t} \big) \big( 1-{r\over t} \big)
\\
&+ P^{ab0} \big( - {x_a \over t} \big) \big( - {x_b \over t} \big) \big( 1-{r\over t} \big) 
\Big),
\endaligned
$$
Then the facts
$$
{r\over t} \leq 1,
\qquad
\big|1-{r^3\over t^3}\big| + \big|1-{r^2\over t^2}\big| + \big|1-{r\over t}\big|
\lesssim {t-r \over t}
\lesssim {s^2 \over t^2}
$$
yield that
$$
|B_1|
\lesssim
{s^2 \over t^2}\big| \del_t v \del_t \del_t u \big|.
$$

As for $B_2$ part, the facts
$$
{r\over t} \leq 1,
\qquad
\sum_a |\underdel_a u| \lesssim |\del u|
$$
imply that
$$
|B_2|
\lesssim \sum_a \big| \del v \underdel_a \del u  \big| + \sum_a \big| \underdel_a v \del \del u \big| + {1\over t} \big| \del v \del u \big|.
$$

In a similar (and easier) way, we can get the desired bound for $P^{\gamma \alpha\beta} \del_\gamma v \del_\alpha u \del_\beta w$.
\end{proof}

The following result tells us that acting some vector fields to the null forms still leads to null forms, and one refers to \cite{Hormander} for more details.

\begin{lemma}\label{lem:comm1}
Let $N(v, u) \footnote{$N_d (v, u) = P_d^{\gamma \alpha\beta} \del_\gamma v \del_\alpha \del_\beta u$, with $P_d^{\gamma \alpha\beta}$ null and $N_0 (v, u) = N (v, u)$. } = P^{\gamma \alpha\beta} \del_\gamma v \del_\alpha \del_\beta u$, then we have ($\Gamma \in \{ \del_0, \del_1, \del_2, \Omega_{12}, L_0, L_1, L_2 \}$)
\be
\aligned
&\Gamma^I (P^{\gamma \alpha\beta} \del_\gamma v \del_\alpha \del_\beta u)
=
\sum_{|I_1| + |I_2| + d = |I|} N_d (\Gamma^{I_1} v, \Gamma^{I_2} u).
\endaligned
\ee
\end{lemma}

We will also need to frequently use the following estimates for commutators, which can be found in \cite{Sogge, PLF-YM-book}.

\begin{lemma} \label{lem:comm2}
Let $u$ be a sufficiently nice function supported in $\Kcal = \{ (t, x) : t \geq |x| + 1\}$, then the following inequalities are valid ($a, b, c \in \{1, 2\}, \alpha, \beta \in \{0, 1, 2\}$)
\be 
\aligned
\big|\del_\alpha L_a u \big|
&\lesssim
\big| L_a \del_\alpha u \big| + \sum_\beta \big| \del_\beta u \big|,
\\
\big| L_a L_b u \big|
&\lesssim
\big| L_b L_a u \big| + \sum_c \big| L_c u \big|,
\\
\big| L_0 L_a u \big|
&\lesssim
\big| L_a L_0 u \big|,
\\
\big| \del_\alpha (u \,s/t) \big|
&\lesssim
\big| (s/t) \del_\alpha u \big| + s^{-1} |u|
\\
\big| L_\alpha (u \,s/t) \big|
&\lesssim
\big| (s/t) L_\alpha u \big| +  \big| (s/t) u \big|,
\\
\big| L_b L_a (u \,s/t) \big| 
&\lesssim
\big| (s/t) L_b L_a u \big| 
+ \big| (s/t) u \big| + \sum_c \big| (s/t) L_c u \big|.
\endaligned
\ee
\end{lemma}

\subsection{Sobolev--type inequality}

In order to obtain wave decay, we need the following Sobolev-type inequality, whose proof can be found in \cite{PLF-YM-book}. The rotation vector field $\Omega_{12}$ is not needed in the inequality \eqref{eq:Sobolev2}, which can be seen from the fact that it can be bounded by the Lorentz boosts within the domain of interest $\{ (t, x) : t \geq |x| + 1 \}$, which reads
$$
\Omega_{12} = t^{-1} \big(x_1 L_2 - x_2 L_1\big).
$$ 

\begin{lemma} \label{lem:sobolev}
Let $u= u(t, x)$ be a sufficiently smooth function with support $\{(t, x): t \geq |x| + 1\}$, then for all  $s \geq 2$, one has 
\bel{eq:Sobolev2}
\sup_{\Hcal_s} \big| t\, u(t, x) \big|  \lesssim \sum_{| J |\leq 2} \| L^J u \|_{L^2_f(\Hcal_s)},
\ee
and the symbol $L$ above represents the Lorentz boosts with $J$ a multi-index. 
\end{lemma}

Together with the commutator estimates in Lemma \ref{lem:comm2}, the following inequality, which can be more conveniently applied, holds
\bel{eq:Sobolev-wave}
\sup_{\Hcal_s} \big| s \hskip0.03cm u(t, x) \big|  \lesssim \sum_{| J |\leq 2} \| (s/t) L^J u \|_{L^2_f(\Hcal_s)},
\ee
with $s \geq 2$.


\subsection{Technical identities and inequalities}\label{subsect:useful}

We prepare some calculations for later use, which will play an important role in the analysis.

\paragraph{Step I.}

\begin{lemma}
Within the cone $\Kcal = \{ (t, x) : t \geq |x| + 1 \}$, we have
\be
\aligned
\del_s \Psi^0_\alpha = -{1\over s} \Psi^0_\alpha + \delta_{0\alpha} {1\over t},
\qquad
|\overline{\del}_a \Psi^0_\alpha | \lesssim {1\over s}.
\endaligned
\ee
\end{lemma}

\begin{proof}

We note
$$
\aligned
\del_s (t/s) = -{t\over s^2} + {1\over t},
\qquad
\del_s (-x_a/s) = {x_a\over s^2},
\\
\overline{\del}_a (t/s) = {1\over s} {x_a \over t},
\qquad
\overline{\del}_a (-x_b/s) = - {\delta_{ab} \over s}.
\endaligned
$$

\end{proof}

\paragraph{Step II.}

\begin{lemma}
Let $w$ be supported in the cone $\Kcal$, satisfying $|\del_s w| \lesssim t^{-1}, |\overline{\del}_a w| \lesssim t^{-1} s^{-1+\delta}$, then it holds
\be
\big| P^{\gamma \alpha\beta} \del_\gamma w \Psi^0_\alpha \Psi^0_\beta \big|
\lesssim
s^{-1+\delta}.
\ee
In addition, we have
\be 
\aligned
&\big| P^{\gamma \alpha\beta} \Psi^0_\gamma \Psi^0_\alpha \Psi^0_\beta \big|
\lesssim {t\over s},
\\
&\big|\del_s \big( P^{\gamma \alpha\beta} \Psi^0_\gamma \Psi^0_\alpha \Psi^0_\beta \big) \big|
\lesssim
{t \over s^2},
\\
&\big|\overline{\del}_a \big( P^{\gamma \alpha\beta} \Psi^0_\gamma \Psi^0_\alpha \Psi^0_\beta \big) \big|
\lesssim
s^{-1}.
\endaligned
\ee
\end{lemma}

\begin{proof}

We find
$$
\aligned
P^{\gamma \alpha\beta} \del_\gamma w \Psi^0_\alpha \Psi^0_\beta
=
P^{\gamma \alpha\beta} \Psi^\eta_\gamma \Psi^0_\alpha \Psi^0_\beta \overline{\del}_\eta w
&=
P^{\gamma \alpha\beta} \Psi^0_\gamma \Psi^0_\alpha \Psi^0_\beta \overline{\del}_s w
+
P^{\gamma \alpha\beta} \Psi^a_\gamma \Psi^0_\alpha \Psi^0_\beta \overline{\del}_a w
\\
&=
P^{\gamma \alpha\beta} \Psi^0_\gamma \Psi^0_\alpha \Psi^0_\beta \overline{\del}_s w
+
P^{a \alpha\beta}  \Psi^0_\alpha \Psi^0_\beta \overline{\del}_a w.
\endaligned
$$
We have
$$
\aligned
P^{\gamma \alpha\beta} \Psi^0_\gamma \Psi^0_\alpha \Psi^0_\beta
&=
P^{000} \Psi^0_0 \Psi^0_0 \Psi^0_0
+
P^{00a} \Psi^0_0 \Psi^0_0 \Psi^0_a
+
P^{0a0} \Psi^0_0 \Psi^0_a \Psi^0_0
+
P^{0ab} \Psi^0_0 \Psi^0_a \Psi^0_b
\\
&+
P^{a00} \Psi^0_a \Psi^0_0 \Psi^0_0
+
P^{a0b} \Psi^0_a \Psi^0_0 \Psi^0_b
+
P^{ab0} \Psi^0_a \Psi^0_b \Psi^0_0
+
P^{abc} \Psi^0_a \Psi^0_b \Psi^0_c
\\
&=
P^{000} {r^3\over s^3} + P^{000} \big( (\Psi^0_0)^3 - {r^3\over s^3} \big)
+
P^{00a} {r^2\over s^2} \Psi^0_a + P^{00a} \big( (\Psi^0_0)^2 - {r^2\over s^2} \big) \Psi^0_a
\\
&+
P^{0a0} {r^2\over s^2} \Psi^0_a + P^{0a0} \big( (\Psi^0_0)^2 - {r^2\over s^2} \big) \Psi^0_a
+
P^{0ab} {r\over s} \Psi^0_a \Psi^0_b + P^{0ab} \big(\Psi^0_0 - {r\over s} \big) \Psi^0_a \Psi^0_b
\\
&+
P^{a00} {r^2 \over s^2} \Psi^0_a  + P^{a00} \big( (\Psi^0_0)^2 - {r^2\over s^2} \big) \Psi^0_a
+
P^{a0b} {r\over s}  \Psi^0_a  \Psi^0_b + P^{a0b} \big(\Psi^0_0 - {r\over s} \big) \Psi^0_a \Psi^0_b
\\
&+
P^{ab0} {r\over s}  \Psi^0_a  \Psi^0_b + P^{ab0} \big(\Psi^0_0 - {r\over s} \big) \Psi^0_a \Psi^0_b
+
P^{abc} \Psi^0_a \Psi^0_b \Psi^0_c,
\endaligned
$$
and since $P^{\gamma \alpha\beta}$ is null, we get
\bel{eq:P-null}
\aligned
P^{\gamma \alpha\beta} \Psi^0_\gamma \Psi^0_\alpha \Psi^0_\beta
&=
P^{000} \big( (\Psi^0_0)^3 - {r^3\over s^3} \big)
+
P^{00a}  \Psi^0_a
+
P^{0a0}  \Psi^0_a
+
P^{0ab} \big(\Psi^0_0 - {r\over s} \big) \Psi^0_a \Psi^0_b
\\
&+
P^{a00}  \Psi^0_a
+
P^{a0b} \big(\Psi^0_0 - {r\over s} \big) \Psi^0_a \Psi^0_b
+
P^{ab0} \big(\Psi^0_0 - {r\over s} \big) \Psi^0_a \Psi^0_b.
\endaligned
\ee
Recall that
$$
\Psi^0_\alpha = - {x_\alpha \over s},
$$
then using the fact that $P^{\gamma \alpha\beta}$ is null, we get
$$
\big| P^{\gamma \alpha\beta} \Psi^0_\gamma \Psi^0_\alpha \Psi^0_\beta \big|
\lesssim {t\over s}.
$$
If $|\del_s w| \lesssim t^{-1}, |\overline{\del}_a w| \lesssim t^{-1} s^{-1+\delta}$, then we have
$$
\big| P^{\gamma \alpha\beta} \del_\gamma w \Psi^0_\alpha \Psi^0_\beta \big|
\lesssim
s^{-1+\delta}.
$$
Furthermore, we act $\overline{\del}_\alpha$ on \eqref{eq:P-null} to get
$$
\big|\del_s \big( P^{\gamma \alpha\beta} \Psi^0_\gamma \Psi^0_\alpha \Psi^0_\beta \big) \big|
\lesssim
{t \over s^2},
\qquad
\big|\overline{\del}_a \big( P^{\gamma \alpha\beta} \Psi^0_\gamma \Psi^0_\alpha \Psi^0_\beta \big) \big|
\lesssim
s^{-1}.
$$
\end{proof}

\paragraph{Step III.}

Similar to the last lemma, we also have the following estimates.

\begin{lemma}
With $\Phi = \big( 1, \pm x_1/t, \pm x_2/t \big)$ one has 
$$
\aligned
|P^{\gamma \alpha\beta} \Phi_\gamma \Phi_\alpha \Phi_\beta |
\lesssim
{s^2 \over t^2}. 
\endaligned
$$ 
\end{lemma}

\paragraph{Step IV.}

\begin{lemma}
Consider the quasilinear wave equation
$$
-\Box u + P^{\gamma \alpha\beta} \del_\gamma w \del_\alpha\del_\beta u = h,
$$
and assume $|\del w| \leq {1\over 100} s^{-1}, |\underdel_a w| \leq {1\over 100} t^{-1} s^{-1+\delta}$,
then we have
\bel{eq:quasi-EE} 
\aligned
\aligned
E (u, s)
&\lesssim
E (u, s_0)
+
\Big|\int_{s_0}^s \int_{\Hcal_\tau^*} 2 (\tau/t) P^{\gamma \alpha \beta} \del_\gamma \del_\alpha w \del_\beta u \del_t u 
\\
&\hskip2.3cm
- (\tau/t) P^{\gamma \alpha\beta} \del_t \del_\gamma w \del_\alpha u \del_\beta u 
+ 2 (\tau/t)  h \del_t u  \, dxd\tau \Big|.
\endaligned
\endaligned
\ee
\end{lemma}

\begin{proof}

Multiplying the equation with $\del_t u$, we have
$$
\aligned
&{1\over 2} \del_t \big( (\del_t u)^2 + \sum_a (\del_a u)^2 \big)
- \del_a \big( \del^a u \del_t u \big)
+ P^{\gamma \alpha \beta} \del_\alpha \big( \del_\gamma w \del_\beta u \del_t u \big)
- {1\over 2} P^{\gamma \alpha\beta} \del_t \big( \del_\gamma w \del_\alpha u \del_\beta u \big)
\\
=
&P^{\gamma \alpha\beta} \del_\gamma \del_\alpha w \del_\beta u \del_t u
- {1\over 2} P^{\gamma \alpha\beta} \del_t \del_\gamma w \del_\alpha u \del_\beta u
+ h \del_t u.
\endaligned
$$

Thus we have the energy estimates
$$
\aligned
&\int_{\Hcal_s^*} |\overline{\del} u|^2 + 2 P^{\gamma \alpha\beta} \del_\gamma w \del_\beta u \del_t u n_\alpha - P^{\gamma \alpha \beta} \del_\gamma w \del_\alpha u \del_\beta u \, dx
\\
-
&\int_{\Hcal_{s_0}^*} |\overline{\del} u|^2 + 2 P^{\gamma \alpha\beta} \del_\gamma w \del_\beta u \del_t u n_\alpha - P^{\gamma \alpha \beta} \del_\gamma w \del_\alpha u \del_\beta u \, dx
\\
=
&\int_{s_0}^s \int_{\Hcal_\tau^*} 2 (\tau/t) P^{\gamma \alpha \beta} \del_\gamma \del_\alpha w \del_\beta u \del_t u \, dxd\tau
-
\int_{s_0}^s \int_{\Hcal_\tau^*} (\tau/t) P^{\gamma \alpha\beta} \del_t \del_\gamma w \del_\alpha u \del_\beta u \, dxd\tau
\\
+
&\int_{s_0}^s \int_{\Hcal_\tau^*} 2 (\tau/t)  h \del_t u  \, dxd\tau,
\endaligned
$$
in which 
$$
n = (1, -x_a/t).
$$
We observe that
$$
n_\alpha = {s\over t} \del_\alpha s.
$$
Since $P^{\gamma \alpha\beta}$ is null, and $|\del w| \leq {1\over 100} s^{-1}, |\underdel_a w| \leq {1\over 100} t^{-1} s^{-1+\delta}$, thus we have
$$
\big| \int_{\Hcal_s^*} 2 P^{\gamma \alpha\beta} \del_\gamma w \del_\beta u \del_t u n_\alpha - P^{\gamma \alpha \beta} \del_\gamma w \del_\alpha u \del_\beta u \, dx \big|
\leq
{1\over 2} \int_{\Hcal_s^*} |\overline{\del} u|^2 \, dx,
$$
which yields the desired estimates.
\end{proof}


\section{Proof of the boundedness of the energy} 
\label{sec:proof}

\paragraph{Bootstrap assumptions and direct consequences}

We take the following bootstrap assumptions on $[s_0=2, s_1)$:
\bel{eq:BA} 
\aligned
E (\Gamma^I w, s)^{1/2}
&\leq C_1 \eps,
\qquad
&&|I| \leq N,
\\
E_{con} (\Gamma^I w, s)^{1/2}
&\leq C_1 \eps s^\delta,
\qquad
&&|I| \leq N-1,
\\
E_{con} (\Gamma^I w, s)^{1/2}
&\leq C_1 \eps s^{2 \delta},
\qquad
&&|I| \leq N,
\endaligned
\ee
with $0< \delta \ll 1$, $C_1 \gg 1$ some large constant to be determined later, and $\eps \ll 1$ the size of the initial data, and
\bel{eq:s1}
s_1 := \sup \{ s>s_0 : \eqref{eq:BA}~holds\}.
\ee

\begin{lemma}\label{lem:BA0}
Under the assumptions in \eqref{eq:BA}, for all $s \in [s_0, s_1)$ it is true that
\be
\aligned
\big\| (s/t) \del \Gamma^I w \big\|_{L^2_f(\Hcal_s)} 
+s^{-\delta}  \big\| (s/t) \Gamma \Gamma^I w \big\|_{L^2_f(\Hcal_s)}
&\lesssim C_1 \eps,
\qquad
|I| \leq N-1,
\\
\big\| (s/t) \del \Gamma^I w \big\|_{L^2_f(\Hcal_s)} 
+s^{-2 \delta} \big\| (s/t) \Gamma \Gamma^I w \big\|_{L^2_f(\Hcal_s)}
&\lesssim C_1 \eps,
\qquad
|I| \leq N.
\endaligned
\ee
\end{lemma}

\begin{proof}
The proof follows from the definition of the natural energy $E (w, s)$ and the conformal energy $E_{con} (w, s)$ as well as the estimates in Lemma \ref{lem:EE} and Lemma \ref{lem:L2type}.
\end{proof}

The use of the Sobolev inequality in Lemma \ref{lem:sobolev} (more precisely \eqref{eq:Sobolev-wave}) leads us to the following pointwise decay for the solution $w$.

\begin{lemma}\label{lem:BA1}
Let the bootstrap assumptions in \eqref{eq:BA} hold, then for all $s \in [s_0, s_1)$ we have
\bel{eq:decay2}
\aligned
\big| \del \Gamma^I w \big| 
&\lesssim C_1 \eps s^{-1},
\qquad
&|I| \leq N-2,
\\
\sum_a \big| \overline{\del}_a \Gamma^I w \big|  
&\lesssim C_1 \eps t^{-1} s^{-1 + \delta},
\qquad
&|I| \leq N-3,
\\
\sum_a \big| \overline{\del}_a \Gamma^I w \big|  
&\lesssim C_1 \eps t^{-1} s^{-1 + 2 \delta},
\qquad
&|I|  \leq N-2.
\endaligned
\ee
\end{lemma}

\begin{proposition}\label{prop:improved1}
Under the assumptions in \eqref{eq:BA}, we have the following improved estimates for all $s \in [s_0, s_1)$
\bel{eq:improved1} 
\aligned
E(\Gamma^I w, s)^{1/2}
&\lesssim \eps + (C_1 \eps)^2,
\qquad
&&& |I|  \leq N-1,
\\
E_{con}(\Gamma^I w, s)^{1/2}
&\lesssim \eps + (C_1 \eps)^2 s^{\delta},
\qquad
&&&
|I|  \leq N-1.
\endaligned
\ee
\end{proposition}

\begin{proof}
We first act the vector field $\Gamma^I$ (with $|I| \leq N-1$) on the wave equation \eqref{eq:model-2d}, and arrive at
$$
- \Box \Gamma^I w = - \Gamma^I N(w, w)  = - \sum_{|I_1| + |I_2| + d = |I|} N_d (\Gamma^{I_1} w, \Gamma^{I_2} w).
$$
In order to refine the bound for $E(\Gamma^I w, s)^{1/2}$, we apply the energy estimates \eqref{eq:E-E1} to get
$$
\aligned
&E(\Gamma^I w, s)^{1/2}
\\
\lesssim
&E(\Gamma^I w, s_0)^{1/2}
+
\int_{s_0}^s \big\| \Gamma^I N (w, w) \big\|_{L^2_f(\Hcal_\tau)} \, d\tau.
\endaligned
$$
We find  
$$
\aligned
\big\| \Gamma^I N (w, w) \big\|_{L^2_f(\Hcal_\tau)}
&\lesssim
\sum_{|I_1|  + |I_2|  + d = |I| } \big\| N_d (\Gamma^{I_1} w, \Gamma^{I_2} w) \big\|_{L^2_f(\Hcal_\tau)}
\\
&\lesssim
\sum_{|I_1| +  |I_2|  \leq |I| } \Big(\big\| (\tau/t)^2 \del_t \Gamma^{I_1} w \del_t \del_t \Gamma^{I_2} w \big\|_{L^2_f(\Hcal_\tau)}
+ \sum_a \big\| \del \Gamma^{I_1} w \overline{\del}_a \del \Gamma^{I_2} w  \big\|_{L^2_f(\Hcal_\tau)}
\\
&+ \sum_a \big\| \overline{\del}_a \Gamma^{I_1} w \del \del \Gamma^{I_2} w  \big\|_{L^2_f(\Hcal_\tau)}
+  \big\| t^{-1} \del \Gamma^{I_1} w  \del \Gamma^{I_2} w  \big\|_{L^2_f(\Hcal_\tau)} \Big).
\endaligned
$$
Now we estimate each of the terms. First, we have
$$
\aligned
& \sum_{|I_1| +  |I_2|  \leq |I| } \big\| (\tau/t)^2 \del_t \Gamma^{I_1} w \del_t \del_t \Gamma^{I_2} w \big\|_{L^2_f(\Hcal_\tau)}
\\
&\lesssim
\sum_{|I_1| \leq N-3, |I_2| \leq N-1} \big\| (\tau/t) (t-|x|)^{-1} \del_t \Gamma^{I_1} w \big\|_{L^\infty(\Hcal_\tau)}  \big\| (\tau/t) (t-|x|) \del_t \del_t \Gamma^{I_2} w \big\|_{L^2_f(\Hcal_\tau)}
\\
&+
\sum_{|I_1| \leq N-1, |I_2| \leq N-3} \big\| (\tau/t) \del_t \Gamma^{I_1} w \big\|_{L^2_f(\Hcal_\tau)}  \big\| (\tau/t) \del_t \del_t \Gamma^{I_2} w \big\|_{L^\infty(\Hcal_\tau)}
\\
&\lesssim
\sum_{|I_1| \leq N-3, |I_2| \leq N-1} \big\| (\tau/t) (t-|x|)^{-1} \del_t \Gamma^{I_1} w \big\|_{L^\infty(\Hcal_\tau)}  \big\| (\tau/t) \Gamma \del_t \Gamma^{I_2} w \big\|_{L^2_f(\Hcal_\tau)}
\\
&+
\sum_{|I_1| \leq N-1, |I_2| \leq N-3} \big\| (\tau/t) \del_t \Gamma^{I_1} w \big\|_{L^2_f(\Hcal_\tau)}  \big\| (\tau/t) (t-|x|)^{-1} \Gamma \del_t \Gamma^{I_2} w \big\|_{L^\infty(\Hcal_\tau)}
\lesssim
(C_1 \eps)^2 \tau^{-2},
\endaligned
$$
in which we used the relation $|\del w| \lesssim (t-|x|)^{-1} |\Gamma w|$.
Next, we proceed by estimating
$$
\aligned
& 
\sum_{a, |I_1| + |I_2| \leq N-1} \big\| \del \Gamma^{I_1} w \overline{\del}_a \del \Gamma^{I_2} w  \big\|_{L^2_f(\Hcal_\tau)}
\\
&\lesssim
\sum_{a, |I_1| + |I_2| \leq N-1} \big\| \del \Gamma^{I_1} w t^{-1} L_a \del \Gamma^{I_2} w  \big\|_{L^2_f(\Hcal_\tau)}
\\
&\lesssim
\sum_{a, |I_1|\leq N-1, |I_2| \leq N-3} \big\| (\tau/t) \del \Gamma^{I_1} w \big\|_{L^2_f(\Hcal_\tau)} \big\| (t/\tau) t^{-1} L_a \del \Gamma^{I_2} w  \big\|_{L^\infty(\Hcal_\tau)}
\\
&+
\sum_{a, |I_1|\leq N-3, |I_2| \leq N-1} \big\| \del \Gamma^{I_1} w t^{-1} (t/\tau) \big\|_{L^\infty(\Hcal_\tau)} \big\| (\tau/t) L_a \del \Gamma^{I_2} w  \big\|_{L^2_f(\Hcal_\tau)}
\lesssim
(C_1 \eps)^2 \tau^{-2}.
\endaligned
$$
Then, we note
$$
\aligned
& 
\sum_{a, |I_1| + |I_2| \leq N-1} \big\| \overline{\del}_a \Gamma^{I_1} w \del \del \Gamma^{I_2} w  \big\|_{L^2_f(\Hcal_\tau)}
\\
&\lesssim
\sum_{a, |I_1| \leq N-3, |I_2| \leq N-1} \big\| \overline{\del}_a \Gamma^{I_1} w (t/\tau) \big\|_{L^\infty(\Hcal_\tau)} \big\| (\tau/t)  \del \del \Gamma^{I_2} w  \big\|_{L^2_f(\Hcal_\tau)}
\\
&+
\sum_{a, |I_1| \leq N-1, |I_2| \leq N-3} \big\| \tau \overline{\del}_a \Gamma^{I_1} w  \big\|_{L^2_f(\Hcal_\tau)} \big\| \tau^{-1} \del \del \Gamma^{I_2} w  \big\|_{L^\infty (\Hcal_\tau)}
\lesssim (C_1 \eps)^2 \tau^{-2+\delta}.
\endaligned
$$
Finally, we easily get (without details)
$$
\aligned
\sum_{|I_1| + |I_2| \leq N-1} \big\| t^{-1} \del \Gamma^{I_1} w  \del \Gamma^{I_2} w  \big\|_{L^2_f(\Hcal_\tau)}
\lesssim (C_1 \eps)^2 \tau^{-2}.
\endaligned
$$
To sum things up, we arrive at
\bel{eq:low-order}
\aligned
\big\| \Gamma^I N (w, w) \big\|_{L^2_f(\Hcal_\tau)}
\lesssim (C_1 \eps)^2 \tau^{-2+\delta},
\qquad
|I| \leq N-1.
\endaligned
\ee
Then the energy estimates give us
$$
\aligned
E(\Gamma^I w, s)^{1/2}
\lesssim
&E(\Gamma^I w, s_0)^{1/2}
+
\int_{s_0}^s \big\| \Gamma^I N (w, w) \big\|_{L^2_f(\Hcal_{\tau})} \, d\tau
\\
\lesssim
& \eps + (C_1 \eps)^2 \int_{s_0}^s \tau^{-2+\delta} \, d\tau
\lesssim
\eps + (C_1 \eps)^2,
\qquad
|I| \leq N-1.
\endaligned
$$
Similarly, the conformal energy estimates lead us to
$$
\aligned
E_{con}(\Gamma^I w, s)^{1/2}
\lesssim
&E_{con}(\Gamma^I w, s_0)^{1/2}
+
\int_{s_0}^s \tau \big\| \Gamma^I N (w, w) \big\|_{L^2_f(\Hcal_{\tau})} \, d\tau
\\
\lesssim
& \eps + (C_1 \eps)^2 \int_{s_0}^s \tau^{-1+\delta} \, d\tau
\lesssim
\eps + (C_1 \eps)^2 s^\delta,
\qquad
|I| \leq N-1,
\endaligned
$$
in which we used \eqref{eq:low-order}. 
\end{proof}

\begin{proposition}\label{prop:improved2}
Under the assumptions in \eqref{eq:BA}, we have the following improved estimates for all $s \in [s_0, s_1)$
\bel{eq:improved2} 
\aligned
E(\Gamma^I w, s)^{1/2}
&\lesssim \eps + (C_1 \eps)^{3/2},
\qquad
|I| \leq N.
\endaligned
\ee
\end{proposition}

\begin{proof}
Applying the vector field $\Gamma^I$ (with $|I| = N$) to the wave equation \eqref{eq:model-2d}, we have
$$
- \Box \Gamma^I w = \Gamma^I N(w, w)  = \sum_{|I_1| + |I_2| + d = |I|} N_d (\Gamma^{I_1} w, \Gamma^{I_2} w),
$$
which is
\bel{eq:quasi-div}
-\Box \Gamma^I w + P^{\gamma \alpha\beta} \del_\gamma w \del_\alpha \del_\beta \Gamma^I w 
= - \sum_{\substack{|I_1| + |I_2| + d = |I|\\ d\geq 1}} N_d (\Gamma^{I_1} w, \Gamma^{I_2} w) - \sum_{\substack{|I_1| + |I_2|  = |I|\\ |I_2| \leq N-1}} N (\Gamma^{I_1} w, \Gamma^{I_2} w).
\ee
Recall the energy estimates for quasilinear wave, and we have
$$
\aligned
E(\Gamma^I w, s)
&\lesssim 
E(\Gamma^I w, s_0)
+
\int_{s_0}^s \int_{\Hcal_{\tau}^*} \Big| (\tau/t) P^{\gamma \alpha\beta} \del_\gamma \del_\alpha w \del_\beta \Gamma^I w \del_t \Gamma^I w \Big| \, dxd\tau
\\
&+
\int_{s_0}^s \int_{\Hcal_{\tau}^*} \Big| (\tau/t) P^{\gamma \alpha\beta} \del_t \del_\gamma w \del_\alpha \Gamma^I w  \del_\beta \Gamma^I w \Big| \, dxd\tau
\\
&+
\int_{s_0}^s \int_{\Hcal_{\tau}^*} \Big| (\tau/t)   \Big(   \sum_{\substack{|I_1| + |I_2| + d = |I|\\ d\geq 1}} N_d (\Gamma^{I_1} w, \Gamma^{I_2} w) + \sum_{\substack{|I_1| + |I_2|  = |I|\\ |I_2| \leq N-1}} N (\Gamma^{I_1} w, \Gamma^{I_2} w)   \Big)   \del_t \Gamma^I w \Big| \, dxd\tau.
\endaligned
$$
First, we note that
$$
\aligned
& 
\big\| P^{\gamma \alpha\beta} \del_\gamma \del_\alpha w \del_\beta \Gamma^I w \big\|_{L^2_f(\Hcal_{\tau})}
\\
&\lesssim
\big\| (\tau/t)^2 \del_t \del_t w \del_t \Gamma^I w  \big\|_{L^2_f(\Hcal_{\tau})}
+
\sum_a \big\| \del_t \underdel_a w \del_t \Gamma^I w  \big\|_{L^2_f(\Hcal_{\tau})}
\\
&+
\sum_a \big\| \del_t \del_t w \underdel_a \Gamma^I w  \big\|_{L^2_f(\Hcal_{\tau})}
+
\sum_{a,b} \big\| \del_t \underdel_a w \underdel_b \Gamma^I w  \big\|_{L^2_f(\Hcal_{\tau})}
\\
&\lesssim
\big\| (\tau/t) \del_t \del_t w \big\|_{L^\infty(\Hcal_{\tau})} \big\| (\tau/t) \del_t \Gamma^I w  \big\|_{L^2_f(\Hcal_{\tau})}
+
\sum_a \big\| \del_t \underdel_a w \big\|_{L^\infty(\Hcal_{\tau})} \big\|  \del_t \Gamma^I w  \big\|_{L^2_f(\Hcal_{\tau})}
\\
&+
\sum_a \big\| \tau^{-1} \del_t \del_t w \big\|_{L^\infty(\Hcal_{\tau})} \big\| \tau \underdel_a \Gamma^I w  \big\|_{L^2_f(\Hcal_{\tau})}
+
\sum_{a,b} \big\| \tau^{-1} \del_t \underdel_a w \big\|_{L^\infty(\Hcal_{\tau})} \big\| \tau \underdel_b \Gamma^I w  \big\|_{L^2_f(\Hcal_{\tau})}
\\
&\lesssim (C_1 \eps)^2 \tau^{-2+ 2 \delta},
\endaligned
$$
which leads us to
$$
\int_{s_0}^s \int_{\Hcal_{\tau}^*} \Big| (\tau/t) P^{\gamma \alpha\beta} \del_\gamma \del_\alpha w \del_\beta \Gamma^I w \del_t \Gamma^I w \Big| \, dxd\tau
\lesssim
(C_1 \eps)^2 \int_{s_0}^s  \tau^{-2+ 2 \delta} \, d\tau
\lesssim (C_1 \eps)^3.
$$

Next, we find that
$$
\aligned
& 
\big\| (\tau/t) P^{\gamma \alpha\beta} \del_t \del_\gamma w \del_\alpha \Gamma^I w  \del_\beta \Gamma^I w \big\|_{L^1_f(\Hcal_{\tau})}
\\
&\lesssim
\big\| (\tau/t)^3  \del_t \del_t w \del_t \Gamma^I w  \del_t \Gamma^I w \big\|_{L^1_f(\Hcal_{\tau})}
+
\big\| (\tau/t) \underdel_a \del_t w \del_t \Gamma^I w  \del_t \Gamma^I w \big\|_{L^1_f(\Hcal_{\tau})}
\\
&+
\big\| (\tau/t) \del_t \del_t w \underdel_a \Gamma^I w  \del_t \Gamma^I w \big\|_{L^1_f(\Hcal_{\tau})}
+
\big\| (\tau/t) \del_t \del_t w \del_t \Gamma^I w  \underdel_a \Gamma^I w \big\|_{L^1_f(\Hcal_{\tau})}
\\
&\lesssim
\big\| (\tau/t)  \del_t \del_t w \big\|_{L^\infty(\Hcal_{\tau})} \big\| (\tau/t) \del_t \Gamma^I w \big\|_{L^2_f(\Hcal_{\tau})} \big\| (\tau/t) \del_t \Gamma^I w \big\|_{L^2_f(\Hcal_{\tau})}
\\
&+
\big\| (t/\tau) \underdel_a \del_t w \big\|_{L^\infty(\Hcal_{\tau})} \big\| (\tau/t) \del_t \Gamma^I w \big\|_{L^2_f(\Hcal_{\tau})} \big\| (\tau/t) \del_t \Gamma^I w \big\|_{L^2_f(\Hcal_{\tau})}
\\
&+
\big\| \tau^{-1} \del_t \del_t w \big\|_{L^\infty(\Hcal_{\tau})} \big\| \tau \underdel_a \Gamma^I w \big\|_{L^2_f(\Hcal_{\tau})} \big\| (\tau/t) \del_t \Gamma^I w \big\|_{L^2_f(\Hcal_{\tau})}
\\
&+
\big\| \tau^{-1} \del_t \del_t w \big\|_{L^\infty(\Hcal_{\tau})} \big\| (\tau/t) \del_t \Gamma^I w \big\|_{L^2_f(\Hcal_{\tau})} \big\| \tau \underdel_a \Gamma^I w \big\|_{L^2_f(\Hcal_{\tau})} 
\lesssim
(C_1 \eps)^3 \tau^{-2+2\delta},
\endaligned
$$
which implies that
$$
\int_{s_0}^s \int_{\Hcal_{\tau}} \Big| (\tau/t) P^{\gamma \alpha\beta} \del_t \del_\gamma w \del_\alpha \Gamma^I w  \del_\beta \Gamma^I w \Big| \, dxd\tau
\lesssim (C_1 \eps)^3.
$$
Similarly, we obtain
$$
\aligned
 \sum_{\substack{|I_1| + |I_2| + d = |I|\\ d\geq 1}} \big\| N_d (\Gamma^{I_1} w, \Gamma^{I_2} w) \big\|_{L^2_f(\Hcal_{\tau})}  + \sum_{\substack{|I_1| + |I_2|  = |I|\\ |I_2| \leq N-1}} \big\| N (\Gamma^{I_1} w, \Gamma^{I_2} w) \big\|_{L^2_f(\Hcal_{\tau})} 
\lesssim (C_1 \eps)^3 \tau^{-2+2\delta},
\endaligned
$$
and we further have
$$
\aligned
& 
\int_{s_0}^s \int_{\Hcal_{\tau}} \Big| (\tau/t)   \Big(   \sum_{\substack{|I_1| + |I_2| + d = |I|\\ d\geq 1}} N_d (\Gamma^{I_1} w, \Gamma^{I_2} w) + \sum_{\substack{|I_1| + |I_2|  = |I|\\ |I_2| \leq N-1}} N (\Gamma^{I_1} w, \Gamma^{I_2} w)   \Big)   \del_t \Gamma^I w \Big| \, dxd\tau
\\
&\lesssim
(C_1 \eps)^3.
\endaligned
$$
The proof is completed.
\end{proof}

The estimates left to be improved are $E_{con}(\Gamma^I w, s)^{1/2}$ with $|I| = N$. We need the following result, which tells us that the conformal energy $E_{con}(\Gamma^I w, s)$ and $\widetilde{E}_{con}(\Gamma^I w, s)$ in Proposition \ref{prop:quasi-conformal} can be somehow bounded by each other.

\begin{lemma}\label{lem:conformal-equ}
With the assumptions in \eqref{eq:BA}, we have the following estimates for all $s \in [s_0, s_1)$
\bel{eq:conformal-equ}
\big| E_{con}(\Gamma^I w, s) - \widetilde{E}_{con}(\Gamma^I w, s)  \big|
\lesssim
(C_1 \eps)^3 s^{4\delta},
\qquad
|I| = N.
\ee
\end{lemma}

\begin{proof}
We take $u = \Gamma^I w$ with $|I| = N$ in Proposition \ref{prop:quasi-conformal}, and we recall the formulas for $E_{con}(\Gamma^I w, s), \widetilde{E}_{con}(\Gamma^I w, s)$, which read
$$
\aligned
&E_{con} (\Gamma^I w, s)
=
\int_{\Hcal_s^*} \sum_a \big( s \underdel_a \Gamma^I w \big)^2 + \big( s \del_s \Gamma^I w + 2 x^a \overline{\del}_a \Gamma^I w + \Gamma^I w \big)^2  \, dx,
\\
&\widetilde{E}_{con} (\Gamma^I w, s)
=
\int_{\Hcal_{s}^*} {1\over 2} Z^2 + M_2  \, dx,
\\
&Z 
= 
g^{\alpha\beta} \Psi^0_\alpha \Psi^0_\beta s \del_s \Gamma^I w 
+ 2 g^{\alpha\beta} \Psi^0_\alpha \Psi^a_\beta s \overline{\del}_a \Gamma^I w 
+ g^{\alpha\beta} \Psi^\gamma_\alpha (\overline{\del}_\gamma \Psi^0_\beta) s \Gamma^I w
- g^{\alpha\beta} \Psi^ 0_\alpha \Psi^0_\beta \Gamma^I w,
\\
&M_2 
= - {1\over 2}  s^2 g^{\alpha\beta} \Psi^0_\alpha \Psi^0_\beta  g^{a b} \overline{\del}_a \Gamma^I w \overline{\del}_b \Gamma^I w,
\endaligned
$$
in which 
$
g^{\alpha \beta} = -m^{\alpha \beta} + P^{\gamma \alpha\beta} \del_\gamma w.
$
We note that $E_{con}(\Gamma^I w, s)$ can also be expressed with
$$
\aligned
&E_{con}(\Gamma^I w, s)
=
\int_{\Hcal_{s}^*} {1\over 2} \widehat{Z}^2 + \widehat{M_2}  \, dx,
\\
&\widehat{Z} 
= 
\big(-m^{\alpha\beta}\big) \Psi^0_\alpha \Psi^0_\beta s \del_s \Gamma^I w 
+ 2 \big(-m^{\alpha\beta}\big) \Psi^0_\alpha \Psi^a_\beta s \overline{\del}_a \Gamma^I w 
+ \big(-m^{\alpha\beta}\big) \Psi^\gamma_\alpha (\overline{\del}_\gamma \Psi^0_\beta) s \Gamma^I w
\\
&\hskip0.35cm - \big(-m^{\alpha\beta}\big) \Psi^ 0_\alpha \Psi^0_\beta \Gamma^I w,
\\
&\widehat{M_2} 
= - {1\over 2}  s^2 \big(-m^{\alpha\beta}\big) \Psi^0_\alpha \Psi^0_\beta  \big(-m^{a b}\big) \overline{\del}_a \Gamma^I w \overline{\del}_b \Gamma^I w,
\endaligned
$$
Thus we have
\be 
\big| E_{con}(\Gamma^I w, s) - \widetilde{E}_{con}(\Gamma^I w, s)  \big|
\lesssim
\int_{\Hcal_{s}^*} \big|  Z^2 - \widehat{Z}^2  \big| + \big| M_2 - \widehat{M_2} \big|  \, dx.
\ee

First, we note
$$
\aligned
Z - \widehat{Z}
&=
P^{\gamma \alpha\beta} \del_\gamma w  \Psi^0_\alpha \Psi^0_\beta s \del_s \Gamma^I w
+ 2 P^{\gamma \alpha\beta} \del_\gamma w  \Psi^0_\alpha \Psi^a_\beta s \overline{\del}_a \Gamma^I w 
+ P^{\gamma \alpha\beta} \del_\gamma w \Psi^{\gamma'}_\alpha (\overline{\del}_{\gamma'} \Psi^0_\beta) s \Gamma^I w
\\
&- P^{\gamma \alpha\beta} \del_\gamma w \Psi^ 0_\alpha \Psi^0_\beta \Gamma^I w
\\
&=
P^{\gamma \alpha\beta} \Psi^\eta_\gamma  \overline{\del}_\eta w  \Psi^0_\alpha \Psi^0_\beta s \del_s \Gamma^I w 
+ 2 P^{\gamma \alpha\beta} \Psi^\eta_\gamma  \overline{\del}_\eta w   \Psi^0_\alpha \Psi^a_\beta s \overline{\del}_a \Gamma^I w 
+ P^{\gamma \alpha\beta} \Psi^\eta_\gamma  \overline{\del}_\eta w  \Psi^{\gamma'}_\alpha (\overline{\del}_{\gamma'} \Psi^0_\beta) s \Gamma^I w
\\
&- P^{\gamma \alpha\beta} \Psi^\eta_\gamma  \overline{\del}_\eta w  \Psi^ 0_\alpha \Psi^0_\beta \Gamma^I w
\\
&=
P^{\gamma \alpha\beta} \Psi^0_\gamma  \overline{\del}_0 w  \Psi^0_\alpha \Psi^0_\beta s \del_s \Gamma^I w 
+
P^{\gamma \alpha\beta} \Psi^a_\gamma  \overline{\del}_a w  \Psi^0_\alpha \Psi^0_\beta s \del_s \Gamma^I w 
+
2 P^{\gamma \alpha\beta} \Psi^\eta_\gamma  \overline{\del}_\eta w   \Psi^0_\alpha \Psi^a_\beta s \overline{\del}_a \Gamma^I w
\\
&+ P^{\gamma \alpha\beta} \Psi^\eta_\gamma  \overline{\del}_\eta w  \Psi^{0}_\alpha (\overline{\del}_{0} \Psi^0_\beta) s \Gamma^I w
+ P^{\gamma \alpha\beta} \Psi^\eta_\gamma  \overline{\del}_\eta w  \Psi^{a}_\alpha (\overline{\del}_{a} \Psi^0_\beta) s \Gamma^I w
- P^{\gamma \alpha\beta} \Psi^\eta_\gamma  \overline{\del}_\eta w  \Psi^ 0_\alpha \Psi^0_\beta \Gamma^I w.
\endaligned
$$
Next, we estimate each of the terms.
Recall that $| \overline{\del}_0 w| \lesssim C_1 \eps t^{-1}, | P^{\gamma \alpha\beta} \Psi^0_\gamma \Psi^0_\alpha \Psi^0_\beta | \lesssim (t/s)$, and we find
$$
\aligned
\| P^{\gamma \alpha\beta} \Psi^0_\gamma  \overline{\del}_0 w  \Psi^0_\alpha \Psi^0_\beta s \del_s \Gamma^I w \|_{L^2_f(\Hcal_s)}
\lesssim
&\| P^{\gamma \alpha\beta} \Psi^0_\gamma  \overline{\del}_0 w  \Psi^0_\alpha \Psi^0_\beta s \|_{L^\infty(\Hcal_s)} \| \del_s \Gamma^I w \|_{L^2_f(\Hcal_s)}
\\
\lesssim 
& C_1 \eps \| (t/s) t^{-1} s \|_{L^\infty(\Hcal_s)} \| \del_s \Gamma^I w \|_{L^2_f(\Hcal_s)}
\lesssim \big( C_1 \eps \big)^2.
\endaligned
$$
Recall that $| s \overline{\del}_a w | \lesssim C_1 \eps t^{-1} s^\delta$, and we find
$$
\aligned
\| P^{\gamma \alpha\beta} \Psi^a_\gamma  \overline{\del}_a w  \Psi^0_\alpha \Psi^0_\beta s \del_s \Gamma^I w \|_{L^2_f(\Hcal_s)}
\lesssim
&\| P^{\gamma \alpha\beta} \Psi^a_\gamma  \overline{\del}_a w  \Psi^0_\alpha \Psi^0_\beta s \|_{L^\infty(\Hcal_s)} \| \del_s \Gamma^I w \|_{L^2_f(\Hcal_s)}
\\
\lesssim
& C_1 \eps \| t^{-1} s^\delta (t/s)^2 \|_{L^\infty(\Hcal_s)} \| \del_s \Gamma^I w \|_{L^2_f(\Hcal_s)}
\lesssim \big( C_1 \eps \big)^2 s^\delta.
\endaligned
$$
Recall that $| \overline{\del}_\alpha w| \lesssim C_1 \eps t^{-1}$, and we find
$$
\aligned
\|2 P^{\gamma \alpha\beta} \Psi^\eta_\gamma  \overline{\del}_\eta w   \Psi^0_\alpha \Psi^a_\beta s \overline{\del}_a \Gamma^I w \|_{L^2_f(\Hcal_s)}
\lesssim
&\| P^{\gamma \alpha\beta} \Psi^\eta_\gamma  \overline{\del}_\eta w   \Psi^0_\alpha \Psi^a_\beta \|_{L^\infty(\Hcal_s)}  \| s \overline{\del}_a \Gamma^I w \|_{L^2_f(\Hcal_s)}
\\
\lesssim
& C_1 \eps \| (t/s)^2 t^{-1} \|_{L^\infty(\Hcal_s)}  \| s \overline{\del}_a \Gamma^I w \|_{L^2_f(\Hcal_s)}
\lesssim \big( C_1 \eps \big)^2 s^{2\delta}.
\endaligned
$$
In view of $| \overline{\del}_0 w| \lesssim C_1 \eps t^{-1}, | s \overline{\del}_a w | \lesssim C_1 \eps t^{-1} s^\delta, | P^{\gamma \alpha\beta} \Psi^0_\gamma \Psi^0_\alpha \Psi^0_\beta | \lesssim (t/s)$, we find
$$
\aligned
&  \| P^{\gamma \alpha\beta} \Psi^\eta_\gamma  \overline{\del}_\eta w  \Psi^{0}_\alpha (\overline{\del}_{0} \Psi^0_\beta) s \Gamma^I w \|_{L^2_f(\Hcal_s)}
+
\| P^{\gamma \alpha\beta} \Psi^\eta_\gamma  \overline{\del}_\eta w  \Psi^ 0_\alpha \Psi^0_\beta \Gamma^I w \|_{L^2_f(\Hcal_s)}
\\
&\lesssim
\| P^{\gamma \alpha\beta} \Psi^0_\gamma  \overline{\del}_0 w  \Psi^{0}_\alpha (\overline{\del}_{0} \Psi^0_\beta) s \Gamma^I w \|_{L^2_f(\Hcal_s)}
+
\| P^{\gamma \alpha\beta} \Psi^a_\gamma  \overline{\del}_a w  \Psi^{0}_\alpha (\overline{\del}_{0} \Psi^0_\beta) s \Gamma^I w \|_{L^2_f(\Hcal_s)}
\\
&+
\| P^{\gamma \alpha\beta} \Psi^0_\gamma  \overline{\del}_0 w  \Psi^ 0_\alpha \Psi^0_\beta \Gamma^I w \|_{L^2_f(\Hcal_s)}
+
\| P^{\gamma \alpha\beta} \Psi^a_\gamma  \overline{\del}_a w  \Psi^ 0_\alpha \Psi^0_\beta \Gamma^I w \|_{L^2_f(\Hcal_s)}
\\
&\lesssim
\| P^{\gamma \alpha\beta} \Psi^0_\gamma  \overline{\del}_0 w  \Psi^{0}_\alpha s^{-1} \Psi^0_\beta s \Gamma^I w \|_{L^2_f(\Hcal_s)}
+
\| P^{\gamma \alpha\beta} \Psi^0_\gamma  \overline{\del}_0 w  \Psi^{0}_\alpha \delta_{0\beta} t^{-1} s \Gamma^I w \|_{L^2_f(\Hcal_s)}
\\
&+
\| P^{\gamma \alpha\beta} \Psi^a_\gamma  \overline{\del}_a w  \Psi^{0}_\alpha s^{-1} \Psi^0_\beta s \Gamma^I w \|_{L^2_f(\Hcal_s)}
+
\| P^{\gamma \alpha\beta} \Psi^a_\gamma  \overline{\del}_a w  \Psi^{0}_\alpha \delta_{0\beta} t^{-1} s \Gamma^I w \|_{L^2_f(\Hcal_s)}
\\
&+
\| P^{\gamma \alpha\beta} \Psi^0_\gamma  \overline{\del}_0 w  \Psi^ 0_\alpha \Psi^0_\beta \Gamma^I w \|_{L^2_f(\Hcal_s)}
+
\| P^{\gamma \alpha\beta} \Psi^a_\gamma  \overline{\del}_a w  \Psi^ 0_\alpha \Psi^0_\beta \Gamma^I w \|_{L^2_f(\Hcal_s)}, 
\endaligned
$$
which is bounded by 
$$
\aligned
&\lesssim
\| P^{\gamma \alpha\beta} \Psi^0_\gamma  \overline{\del}_0 w  \Psi^{0}_\alpha s^{-1} \Psi^0_\beta s (t/s) \|_{L^\infty(\Hcal_s)}  \| (s/t) \Gamma^I w \|_{L^2_f(\Hcal_s)}
\\
&+
\| P^{\gamma \alpha\beta} \Psi^0_\gamma  \overline{\del}_0 w  \Psi^{0}_\alpha \delta_{0\beta} t^{-1} s (t/s) \|_{L^\infty(\Hcal_s)}  \| (s/t) \Gamma^I w \|_{L^2_f(\Hcal_s)}
\\
&+
\| P^{\gamma \alpha\beta} \Psi^a_\gamma  \overline{\del}_a w  \Psi^{0}_\alpha s^{-1} \Psi^0_\beta s (t/s) \|_{L^\infty(\Hcal_s)}  \| (s/t) \Gamma^I w \|_{L^2_f(\Hcal_s)}
\\
&+
\| P^{\gamma \alpha\beta} \Psi^a_\gamma  \overline{\del}_a w  \Psi^{0}_\alpha \delta_{0\beta} t^{-1} s (t/s) \|_{L^\infty(\Hcal_s)}  \| (s/t) \Gamma^I w \|_{L^2_f(\Hcal_s)}
\\
&+
\| P^{\gamma \alpha\beta} \Psi^0_\gamma  \overline{\del}_0 w  \Psi^ 0_\alpha \Psi^0_\beta (t/s) \|_{L^\infty(\Hcal_s)}  \| (s/t) \Gamma^I w \|_{L^2_f(\Hcal_s)}
\\
&+
\| P^{\gamma \alpha\beta} \Psi^a_\gamma  \overline{\del}_a w  \Psi^ 0_\alpha \Psi^0_\beta (t/s) \|_{L^\infty(\Hcal_s)}  \| (s/t) \Gamma^I w \|_{L^2_f(\Hcal_s)}
\\
&\lesssim C_1 \eps s^\delta \|(s/t) \Gamma^I w \|_{L^2_f(\Hcal_s)}
\lesssim \big( C_1 \eps \big)^2 s^{2\delta}.
\endaligned
$$
Recall that $| \overline{\del}_\alpha w| \lesssim t^{-1}$, and we find
$$
\aligned
&\| P^{\gamma \alpha\beta} \Psi^\eta_\gamma  \overline{\del}_\eta w  \Psi^{a}_\alpha (\overline{\del}_{a} \Psi^0_\beta) s \Gamma^I w \|_{L^2_f(\Hcal_s)}
\\
\lesssim
&\| P^{\gamma \alpha\beta} \Psi^\eta_\gamma  \overline{\del}_\eta w  \Psi^{a}_\alpha (\overline{\del}_{a} \Psi^0_\beta) s (t/s) \|_{L^\infty(\Hcal_s)} \| (s/t) \Gamma^I w \|_{L^2_f(\Hcal_s)}
\\
\lesssim
& C_1 \eps \| (t/s)  t^{-1}  s^{-1} s (t/s) \|_{L^\infty(\Hcal_s)} \| (s/t) \Gamma^I w \|_{L^2_f(\Hcal_s)}
\lesssim \big( C_1 \eps \big)^2 s^{\delta}.
\endaligned
$$
So we obtain
\bel{eq:conformal000}
\int_{\Hcal_{s}^*} \big|  Z^2 - \widehat{Z}^2  \big| \, dx
\lesssim
\big\| Z - \widehat{Z} \big\|_{L^2_f(\Hcal_s)} \big( \big\| Z - \widehat{Z} \big\|_{L^2_f(\Hcal_s)} + 2 \big\| \widehat{Z} \big\|_{L^2_f(\Hcal_s)} \big)
\lesssim
\big( C_1 \eps \big)^3 s^{4\delta}.
\ee
Then, we turn to bound the remaining part, and we get
$$
\aligned
\big| M_2 - \widehat{M_2} \big|
&\lesssim
\big| s^2 m^{\alpha\beta} \Psi^0_\alpha \Psi^0_\beta  P^{\gamma ab} \del_\gamma w \overline{\del}_a \Gamma^I w \overline{\del}_b \Gamma^I w \big|
+
\big| s^2 P^{\gamma \alpha\beta} \del_\gamma w \Psi^0_\alpha \Psi^0_\beta  m^{a b} \overline{\del}_a \Gamma^I w \overline{\del}_b \Gamma^I w \big|
\\
&+
\big| s^2 P^{\gamma \alpha\beta} \del_\gamma w \Psi^0_\alpha \Psi^0_\beta  P^{\gamma' ab} \del_{\gamma'} w \overline{\del}_a \Gamma^I w \overline{\del}_b \Gamma^I w \big|.
\endaligned
$$
Note $m^{\alpha \beta} \Psi^0_\alpha \Psi^0_\beta = -1$, and recall that $ |\del_\alpha w| \lesssim C_1 \eps s^{-1}, |\overline{\del}_\alpha w| \lesssim C_1 \eps t^{-1}$, as well as
$$
\big| P^{\gamma \alpha\beta} \del_\gamma w \Psi^0_\alpha \Psi^0_\beta \big|
\lesssim
\big| P^{\gamma \alpha\beta} \Psi^0_\gamma \overline{\del}_0 w \Psi^0_\alpha \Psi^0_\beta \big|
+
\big| P^{\gamma \alpha\beta} \Psi^a_\gamma \overline{\del}_a w \Psi^0_\alpha \Psi^0_\beta \big|
\lesssim
C_1 \eps,
$$
and in a similar way we find
\bel{eq:conformal001}
\| M_2 - \widehat{M_2} \|_{L^1_f(\Hcal_s)}
\lesssim C_1 \eps \sum_{a} \| s \overline{\del}^a \Gamma^I w \|_{L^2_f(\Hcal_s)}^2
\lesssim \big( C_1 \eps \big)^3 s^{4\delta}.
\ee
Finally, the combination of \eqref{eq:conformal000} and \eqref{eq:conformal001} yields the desired estimates in \eqref{eq:conformal-equ}. 
\end{proof}

\begin{proposition}\label{prop:improved3}
One has 
\bel{eq:improved3}
E_{con}(\Gamma^I w, s)^{1/2}
\lesssim \eps + (C_1 \eps)^{3/2} s^{2\delta},
\qquad
|I|  = N.
\ee
\end{proposition}

\begin{proof}
By the estimates in Lemma \ref{lem:conformal-equ}, it suffices to show 
$$
\widetilde{E}_{con}(\Gamma^I w, s)
\lesssim \eps^2 + (C_1 \eps)^{3} s^{4\delta},
\qquad
|I|  = N.
$$
 
Recall the conformal energy estimates for quasilinear wave equations in Proposition \ref{prop:quasi-conformal} with $u = \Gamma^I$ with $|I| = N$, and we thus have
\be 
\widetilde{E}_{con} (\Gamma^I w, s)
=
\widetilde{E}_{con} (\Gamma^I w, s_0)
+
\int_{s_0}^s \int_{\Hcal_{\tau}^*} \tau Z h - \tau Z M_1 - M_4 \, dx d\tau.
\ee
in which
$$
h
=
- \sum_{\substack{|I_1| + |I_2| + d = N\\ d\geq 1}} N_d (\Gamma^{I_1} w, \Gamma^{I_2} w) - \sum_{\substack{|I_1| + |I_2|  = |I|\\ |I_2| \leq N-1}} N (\Gamma^{I_1} w, \Gamma^{I_2} w).
$$
and
$$
\aligned
Z 
&= 
g^{\alpha\beta} \Psi^0_\alpha \Psi^0_\beta s \del_s \Gamma^I w 
+ 2 g^{\alpha\beta} \Psi^0_\alpha \Psi^a_\beta s \overline{\del}_a \Gamma^I w 
+ g^{\alpha\beta} \Psi^\gamma_\alpha (\overline{\del}_\gamma \Psi^0_\beta) s \Gamma^I w
- g^{\alpha\beta} \Psi^ 0_\alpha \Psi^0_\beta \Gamma^I w,
\\
M_1
&=
 {1\over s} \del_s \big( g^{\alpha\beta} \Psi^ 0_\alpha \Psi^0_\beta \big)  \Gamma^I w   
- \del_s (g^{\alpha\beta} \Psi^0_\alpha \Psi^0_\beta) \del_s \Gamma^I w 
-2 {1\over s} g^{\alpha\beta} \Psi^0_\alpha \Psi^a_\beta \overline{\del}_a \Gamma^I w
\\
& - 2 \del_s (g^{\alpha\beta} \Psi^0_\alpha \Psi^a_\beta ) \overline{\del}_a \Gamma^I w
-{1\over s} g^{\alpha\beta} \Psi^\gamma_\alpha (\overline{\del}_\gamma \Psi^0_\beta) \Gamma^I w
- \del_s \big( g^{\alpha\beta} \Psi^\gamma_\alpha  (\overline{\del}_\gamma \Psi^0_\beta) \big) \Gamma^I w,
\\
M_4
&=
{1\over 2} \del_s \big( s^2 g^{\alpha\beta} \Psi^0_\alpha \Psi^0_\beta  g^{a b} \big) \overline{\del}_a \Gamma^I w \overline{\del}_b \Gamma^I w
- s^2 \overline{\del}_a \big( g^{\alpha\beta} \Psi^0_\alpha \Psi^0_\beta  g^{a b}  \big)  \del_s \Gamma^I w \overline{\del}_b \Gamma^I w
\\
&
- 2 \overline{\del}_{a'} \big( s^2 g^{\alpha\beta} \Psi^0_\alpha \Psi^a_\beta g^{a' b'} \big) \overline{\del}_a \Gamma^I w \overline{\del}_{b'} \Gamma^I w
+ \overline{\del}_a \big( s^2 g^{\alpha\beta} \Psi^0_\alpha \Psi^a_\beta g^{a' b'} \big) \overline{\del}_{a'} \Gamma^I w \overline{\del}_{b'} \Gamma^I w,
\\
&
- \overline{\del}_a \big( s^2 g^{\alpha\beta} \Psi^\gamma_\alpha (\overline{\del}_\gamma \Psi^0_\beta) g^{a b} \big) \Gamma^I w \overline{\del}_b \Gamma^I w 
- s^2 g^{\alpha\beta} \Psi^\gamma_\alpha (\overline{\del}_\gamma \Psi^0_\beta) g^{a b} \overline{\del}_a \Gamma^I w \overline{\del}_b \Gamma^I w
\\
&
+ \overline{\del}_a \big(s g^{\alpha \beta} \Psi^0_\alpha \Psi^0_\beta  g^{ab} \big) \Gamma^I w \overline{\del}_b \Gamma^I w 
+ s g^{\alpha \beta} \Psi^0_\alpha \Psi^0_\beta  g^{ab} \overline{\del}_a \Gamma^I w \overline{\del}_b \Gamma^I w.
\endaligned
$$
From the proof of Proposition \ref{prop:improved2} we know 
$$
\| h \|_{L^2_f (\Hcal_s)}
\lesssim \big( C_1 \eps \big)^2 s^{-2+\delta},
$$
and from the proof of Lemma \ref{lem:conformal-equ} we know
$$
\| Z \|_{L^2_f (\Hcal_s)}
\lesssim C_1 \eps s^{2\delta}, 
$$
so we will only need to take care of the terms $M_1, M_4$.

\

\noindent\textbf{Estimates for $M_1$.}
For the term $M_1$, first note that
$$
\aligned
& 
{1\over s} \del_s \big( m^{\alpha\beta} \Psi^ 0_\alpha \Psi^0_\beta \big)  \Gamma^I w   
- \del_s (m^{\alpha\beta} \Psi^0_\alpha \Psi^0_\beta) \del_s u 
-2 {1\over s} m^{\alpha\beta} \Psi^0_\alpha \Psi^a_\beta \overline{\del}_a \Gamma^I w
- 2 \del_s (m^{\alpha\beta} \Psi^0_\alpha \Psi^a_\beta ) \overline{\del}_a \Gamma^I w
\\
&-{1\over s} m^{\alpha\beta} \Psi^\gamma_\alpha (\overline{\del}_\gamma \Psi^0_\beta) \Gamma^I w
- \del_s \big( m^{\alpha\beta} \Psi^\gamma_\alpha  (\overline{\del}_\gamma \Psi^0_\beta) \big) \Gamma^I w = 0,
\endaligned
$$
and hence we have
$$
\aligned
M_1
&=
M_1
+
{1\over s} \del_s \big( m^{\alpha\beta} \Psi^ 0_\alpha \Psi^0_\beta \big)  \Gamma^I w   
- \del_s (m^{\alpha\beta} \Psi^0_\alpha \Psi^0_\beta) \del_s \Gamma^I w 
-2 {1\over s} m^{\alpha\beta} \Psi^0_\alpha \Psi^a_\beta \overline{\del}_a \Gamma^I w
\\
&- 2 \del_s (m^{\alpha\beta} \Psi^0_\alpha \Psi^a_\beta ) \overline{\del}_a \Gamma^I w
-{1\over s} m^{\alpha\beta} \Psi^\gamma_\alpha (\overline{\del}_\gamma \Psi^0_\beta) \Gamma^I w
- \del_s \big( m^{\alpha\beta} \Psi^\gamma_\alpha  (\overline{\del}_\gamma \Psi^0_\beta) \big) \Gamma^I w
\\
&=
{1\over s} \del_s \big( P^{\gamma\alpha\beta} \del_\gamma w \Psi^ 0_\alpha \Psi^0_\beta \big)  \Gamma^I w   
- \del_s (P^{\gamma\alpha\beta} \del_\gamma w \Psi^0_\alpha \Psi^0_\beta) \del_s \Gamma^I w 
-2 {1\over s} P^{\gamma\alpha\beta} \del_\gamma w \Psi^0_\alpha \Psi^a_\beta \overline{\del}_a \Gamma^I w
\\
&- 2 \del_s (P^{\gamma\alpha\beta} \del_\gamma w \Psi^0_\alpha \Psi^a_\beta ) \overline{\del}_a \Gamma^I w
-{1\over s} P^{\gamma\alpha\beta} \del_\gamma w \Psi^\gamma_\alpha (\overline{\del}_\gamma \Psi^0_\beta) \Gamma^I w
- \del_s \big( P^{\gamma\alpha\beta} \del_\gamma w \Psi^\gamma_\alpha  (\overline{\del}_\gamma \Psi^0_\beta) \big) \Gamma^I w.
\endaligned
$$
Recall that $|\del_s w| \lesssim C_1 \eps t^{-1}$, $|\overline{\del}_a w| \lesssim C_1 \eps s^{-1+\delta} t^{-1}$, $|\del \del w| \lesssim (t-r)^{-1} |\Gamma \del w | \lesssim C_1 \eps (t-r)^{-1} s^{-1}$, and we have
$$
\aligned
&
 \| {1\over s} \del_s \big( P^{\gamma \alpha\beta} \del_\gamma w \Psi^ 0_\alpha \Psi^0_\beta \big)  \Gamma^I w \|_{L^2_f(\Hcal_s)}
\\
&=
\| {1\over s} \del_s \big( P^{\gamma \alpha\beta} \Psi^\eta_\gamma \overline{\del}_\eta w \Psi^ 0_\alpha \Psi^0_\beta \big)  \Gamma^I w \|_{L^2_f(\Hcal_s)}
\\
&\lesssim
\| {1\over s} \del_s \big( P^{\gamma \alpha\beta} \Psi^0_\gamma \overline{\del}_0 w \Psi^ 0_\alpha \Psi^0_\beta \big)  \Gamma^I w \|_{L^2_f(\Hcal_s)}
+
\| {1\over s} \del_s \big( P^{\gamma \alpha\beta} \Psi^a_\gamma \overline{\del}_a w \Psi^ 0_\alpha \Psi^0_\beta \big)  \Gamma^I w \|_{L^2_f(\Hcal_s)}
\\
&\lesssim
C_1 \eps {1\over s^2}  \| (s/t) \Gamma^I w \|_{L^2_f(\Hcal_s)}
+
C_1 \eps {1\over s^{2-\delta}} \| (s/t) \Gamma^I w \|_{L^2_f(\Hcal_s)}
\lesssim \big( C_1 \eps  \big)^2 s^{-2+2\delta}.
\endaligned
$$

Recall next that $|\del_s w| \lesssim C_1 \eps t^{-1}$, $|\overline{\del}_a w| \lesssim C_1 \eps s^{-1+\delta} t^{-1}$, $|\del \del w| \lesssim C_1 \eps (t-r)^{-1} s^{-1}$, and we have
$$
\aligned
&
 \| \del_s (P^{\gamma \alpha\beta} \del_\gamma w \Psi^0_\alpha \Psi^0_\beta) \del_s \Gamma^I w \|_{L^2_f(\Hcal_s)}
\\
&=
\| \del_s (P^{\gamma \alpha\beta} \Psi^\eta_\gamma \overline{\del}_\eta w \Psi^0_\alpha \Psi^0_\beta) \del_s \Gamma^I w \|_{L^2_f(\Hcal_s)}
\\
&\lesssim
\| \del_s (P^{\gamma \alpha\beta} \Psi^0_\gamma \overline{\del}_0 w \Psi^0_\alpha \Psi^0_\beta) \del_s \Gamma^I w \|_{L^2_f(\Hcal_s)}
+
\| \del_s (P^{\gamma \alpha\beta} \Psi^a_\gamma \overline{\del}_a w \Psi^0_\alpha \Psi^0_\beta) \del_s \Gamma^I w \|_{L^2_f(\Hcal_s)}
\\
&\lesssim
C_1 \eps {1\over s^2}  \| \del_s \Gamma^I w \|_{L^2_f(\Hcal_s)}
+
C_1 \eps {1\over s^{2-\delta}} \| \del_s \Gamma^I w \|_{L^2_f(\Hcal_s)}
\lesssim
\big( C_1 \eps  \big)^2 s^{-2+\delta}.
\endaligned
$$
In view of $|\del w| \lesssim C_1 \eps s^{-1}$, we have
$$
\aligned
\| 2 {1\over s} P^{\gamma \alpha\beta} \del_\gamma w \Psi^0_\alpha \Psi^a_\beta \overline{\del}_a \Gamma^I w \|_{L^2_f(\Hcal_s)}
\lesssim
C_1 \eps {1\over s^2} \sum_a \| s \overline{\del}_a \Gamma^I w \|_{L^2_f(\Hcal_s)}
\lesssim \big( C_1 \eps  \big)^2 s^{-2+2\delta}.
\endaligned
$$

Since $|\del w| + |\del \del w| \lesssim s^{-1}$, we have
$$
\aligned
\|2 \del_s (P^{\gamma \alpha\beta} \del_\gamma w \Psi^0_\alpha \Psi^a_\beta ) \overline{\del}_a \Gamma^I w \|_{L^2_f(\Hcal_s)}
\lesssim
C_1 \eps {1\over s^2} \sum_a \| s \overline{\del}_a \Gamma^I w \|_{L^2_f(\Hcal_s)}
\lesssim \big( C_1 \eps  \big)^2 s^{-2+2\delta}.
\endaligned
$$
Recalling that $|\del_s w| \lesssim C_1 \eps t^{-1}$, $|\overline{\del}_a w| \lesssim C_1 \eps s^{-1+\delta} t^{-1}$, we have
$$
\aligned
&
 \| {1\over s} P^{\eta \alpha\beta} \del_\eta w \Psi^\gamma_\alpha (\overline{\del}_\gamma \Psi^0_\beta) \Gamma^I w \|_{L^2_f(\Hcal_s)}
\\
&\leq
\| {1\over s} P^{\eta \alpha\beta} \del_\eta w \Psi^0_\alpha (\overline{\del}_0 \Psi^0_\beta) \Gamma^I w \|_{L^2_f(\Hcal_s)}
+
\| {1\over s} P^{\eta \alpha\beta} \del_\eta w \Psi^a_\alpha (\overline{\del}_a \Psi^0_\beta) \Gamma^I w \|_{L^2_f(\Hcal_s)}
\\
&\leq 
\| {1\over s} P^{\eta \alpha\beta} \Psi^\gamma_\eta  \overline{\del}_\gamma w \Psi^0_\alpha (\overline{\del}_0 \Psi^0_\beta) \Gamma^I w \|_{L^2_f(\Hcal_s)}
+
\| {1\over s} P^{\eta \alpha\beta} \Psi^\gamma_\eta  \overline{\del}_\gamma w \Psi^a_\alpha (\overline{\del}_a \Psi^0_\beta) \Gamma^I w \|_{L^2_f(\Hcal_s)}
\\
&\lesssim
\| {1\over s} P^{\eta \alpha\beta} \Psi^0_\eta  \overline{\del}_0 w \Psi^0_\alpha (\overline{\del}_0 \Psi^0_\beta) \Gamma^I w \|_{L^2_f(\Hcal_s)}
+
\| {1\over s} P^{\eta \alpha\beta} \Psi^a_\eta  \overline{\del}_a w \Psi^0_\alpha (\overline{\del}_0 \Psi^0_\beta) \Gamma^I w \|_{L^2_f(\Hcal_s)}
\\
&+
\| {1\over s} P^{\eta \alpha\beta} \Psi^0_\eta  \overline{\del}_0 w \Psi^a_\alpha (\overline{\del}_a \Psi^0_\beta) \Gamma^I w \|_{L^2_f(\Hcal_s)}
+
\| {1\over s} P^{\eta \alpha\beta} \Psi^b_\eta  \overline{\del}_b w \Psi^a_\alpha (\overline{\del}_a \Psi^0_\beta) \Gamma^I w \|_{L^2_f(\Hcal_s)}
\\
&\lesssim
C_1 \eps {1\over s^2}  \| (s/t) \Gamma^I w \|_{L^2_f(\Hcal_s)}
+
C_1 \eps {1\over s^{2-\delta}} \| (s/t) \Gamma^I w \|_{L^2_f(\Hcal_s)}
\lesssim \big( C_1 \eps  \big)^2 s^{-2+2\delta}.
\endaligned
$$

From $|\del_s w| \lesssim C_1 \eps t^{-1}$, $|\overline{\del}_a w| \lesssim C_1 \eps s^{-1+\delta} t^{-1}$, $|\del \del w| \lesssim C_1 \eps (t-r)^{-1} s^{-1}$, we deduce that 
$$
\aligned
&
 \| \del_s \big( P^{\eta \alpha\beta} \del_\eta w \Psi^\gamma_\alpha  (\overline{\del}_\gamma \Psi^0_\beta) \big) \Gamma^I w \|_{L^2_f(\Hcal_s)}
\\
&\leq 
\| \del_s \big( P^{\eta \alpha\beta} \del_\eta w \Psi^0_\alpha  (\overline{\del}_0 \Psi^0_\beta) \big) \Gamma^I w \|_{L^2_f(\Hcal_s)}
+
\| \del_s \big( P^{\eta \alpha\beta} \del_\eta w \Psi^a_\alpha  (\overline{\del}_a \Psi^0_\beta) \big) \Gamma^I w \|_{L^2_f(\Hcal_s)}
\\
&\leq
\| \del_s \big( P^{\eta \alpha\beta} \Psi^\gamma_\eta  \overline{\del}_\gamma w \Psi^0_\alpha  (\overline{\del}_0 \Psi^0_\beta) \big) \Gamma^I w \|_{L^2_f(\Hcal_s)}
+
\| \del_s \big( P^{\eta \alpha\beta} \Psi^\gamma_\eta  \overline{\del}_\gamma w \Psi^a_\alpha  (\overline{\del}_a \Psi^0_\beta) \big) \Gamma^I w \|_{L^2_f(\Hcal_s)}
\\
&\leq
\| \del_s \big( P^{\eta \alpha\beta} \Psi^0_\eta  \overline{\del}_0 w \Psi^0_\alpha  (\overline{\del}_0 \Psi^0_\beta) \big) \Gamma^I w \|_{L^2_f(\Hcal_s)}
+
\| \del_s \big( P^{\eta \alpha\beta} \Psi^a_\eta  \overline{\del}_a w \Psi^0_\alpha  (\overline{\del}_0 \Psi^0_\beta) \big) \Gamma^I w \|_{L^2_f(\Hcal_s)}
\\
&+
\| \del_s \big( P^{\eta \alpha\beta} \Psi^0_\eta  \overline{\del}_0 w \Psi^a_\alpha  (\overline{\del}_a \Psi^0_\beta) \big) \Gamma^I w \|_{L^2_f(\Hcal_s)}
+
\| \del_s \big( P^{\eta \alpha\beta} \Psi^b_\eta  \overline{\del}_b w \Psi^a_\alpha  (\overline{\del}_a \Psi^0_\beta) \big) \Gamma^I w \|_{L^2_f(\Hcal_s)}
\\
&\lesssim
C_1 \eps {1\over s^2}  \| (s/t) \Gamma^I w \|_{L^2_f(\Hcal_s)}
+
C_1 \eps {1\over s^{2-\delta}} \| (s/t) \Gamma^I w \|_{L^2_f(\Hcal_s)}
\lesssim \big( C_1 \eps  \big)^2 s^{-2+2\delta},
\endaligned
$$
in which we used the observations
$$
\del_s \del_s \Psi^0_\alpha
=
{2\over s^2} \Psi^0_\alpha - \delta_{0\alpha} {1\over s t} - \delta_{0 \alpha} {s\over t^3}
$$
and
$$
\big| \del_s \del_s w \big|
=
\big| {1\over t} \del_t w + {s^2 \over t^2} \del_t \del_t w - {s^2\over t^3} \del_t w \big|
\lesssim C_1 \eps t^{-1} s^{-1},
\qquad
\big| \del_s \overline{\del}_a w \big|
=
\big| {s\over t^2} \del_t L_a w - {s\over t^3} L_a w \big|
\lesssim C_1 \eps t^{-2}.
$$
Gathering the estimates leads us to
\be 
\int_{s_0}^s \int_{\Hcal_{\tau}^*} | \tau Z M_1 | \, dxd\tau
\lesssim
\int_{s_0}^s \tau \| Z \|_{L^2_f(\Hcal_\tau)}   \| M_1 \|_{L^2_f(\Hcal_\tau)} \, d\tau
\lesssim
\big( C_1 \eps \big)^3 \int_{s_0}^s \tau^{-1+2\delta} \, d\tau
\lesssim
\big( C_1 \eps \big)^3 s^{2\delta}.
\ee

\

\noindent\textbf{Estimates for $M_4$.}
Note that 
$$
\aligned
&
{1\over 2} \del_s \big( s^2 (-m^{\alpha\beta}) \Psi^0_\alpha \Psi^0_\beta  (-m^{a b}) \big) \overline{\del}_a u \overline{\del}_b \Gamma^I w
- s^2 \overline{\del}_a \big( (-m^{\alpha\beta}) \Psi^0_\alpha \Psi^0_\beta (- m^{a b} ) \big)  \del_s u \overline{\del}_b \Gamma^I w
\\
&- 2 \overline{\del}_{a'} \big( s^2 (-m^{\alpha\beta}) \Psi^0_\alpha \Psi^a_\beta (-m^{a' b'}) \big) \overline{\del}_a u \overline{\del}_{b'} \Gamma^I w
+ \overline{\del}_a \big( s^2 (-m^{\alpha\beta}) \Psi^0_\alpha \Psi^a_\beta (-m^{a' b'}) \big) \overline{\del}_{a'} u \overline{\del}_{b'} \Gamma^I w,
\\
&- \overline{\del}_a \big( s^2 (-m^{\alpha\beta}) \Psi^\gamma_\alpha (\overline{\del}_\gamma \Psi^0_\beta) (-m^{a b}) \big) u \overline{\del}_b \Gamma^I w 
- s^2 (-m^{\alpha\beta}) \Psi^\gamma_\alpha (\overline{\del}_\gamma \Psi^0_\beta) (-m^{a b}) \overline{\del}_a u \overline{\del}_b \Gamma^I w
\\
&+ \overline{\del}_a \big(s (-m^{\alpha \beta}) \Psi^0_\alpha \Psi^0_\beta  (-m^{ab}) \big) u \overline{\del}_b \Gamma^I w 
+ s (-m^{\alpha \beta}) \Psi^0_\alpha \Psi^0_\beta  (-m^{ab}) \overline{\del}_a u \overline{\del}_b \Gamma^I w
=0,
\endaligned
$$
which means that 
$$
\aligned
M_4
&=
M_4 
- \Big({1\over 2} \del_s \big( s^2 m^{\alpha\beta} \Psi^0_\alpha \Psi^0_\beta  m^{a b} \big) \overline{\del}_a \Gamma^I w \overline{\del}_b \Gamma^I w
- s^2 \overline{\del}_a \big( m^{\alpha\beta} \Psi^0_\alpha \Psi^0_\beta  m^{a b}  \big)  \del_s \Gamma^I w \overline{\del}_b \Gamma^I w
\\
&- 2 \overline{\del}_{a'} \big( s^2 m^{\alpha\beta} \Psi^0_\alpha \Psi^a_\beta m^{a' b'} \big) \overline{\del}_a \Gamma^I w \overline{\del}_{b'} \Gamma^I w
+ \overline{\del}_a \big( s^2 m^{\alpha\beta} \Psi^0_\alpha \Psi^a_\beta m^{a' b'} \big) \overline{\del}_{a'} \Gamma^I w \overline{\del}_{b'} \Gamma^I w,
\\
&- \overline{\del}_a \big( s^2 m^{\alpha\beta} \Psi^\gamma_\alpha (\overline{\del}_\gamma \Psi^0_\beta) m^{a b} \big) \Gamma^I w \overline{\del}_b \Gamma^I w 
- s^2 m^{\alpha\beta} \Psi^\gamma_\alpha (\overline{\del}_\gamma \Psi^0_\beta) m^{a b} \overline{\del}_a \Gamma^I w \overline{\del}_b \Gamma^I w 
\\
&+ \overline{\del}_a \big(s m^{\alpha \beta} \Psi^0_\alpha \Psi^0_\beta  m^{ab} \big) \Gamma^I w \overline{\del}_b \Gamma^I w 
+ s m^{\alpha \beta} \Psi^0_\alpha \Psi^0_\beta  m^{ab} \overline{\del}_a \Gamma^I w \overline{\del}_b \Gamma^I w \Big) 
\endaligned
$$
and, after cancellation, 
$$
\aligned
M_4
= H_1 + H_2 + H_3 + H_4 + H_5 + H_6 + H_7 + H_8,
\endaligned
$$
with
$$
\aligned
H_1
&=
{1\over 2} \Big( \del_s \big( s^2 P^{\gamma \alpha\beta} \del_\gamma w \Psi^0_\alpha \Psi^0_\beta  \big(-m^{a b}\big) \big) 
+ \del_s \big( s^2 \big(-m^{\alpha\beta}\big) \Psi^0_\alpha \Psi^0_\beta  P^{\gamma a b} \del_\gamma w \big)
\\
&+ \del_s \big( s^2 P^{\gamma \alpha\beta} \del_\gamma w \Psi^0_\alpha \Psi^0_\beta  P^{\gamma' a b} \del_{\gamma'} w \big)
\Big) \overline{\del}_a \Gamma^I w \overline{\del}_b \Gamma^I w,
\\
H_2
&=- s^2 \Big( \overline{\del}_a \big( P^{\gamma \alpha\beta} \del_\gamma w \Psi^0_\alpha \Psi^0_\beta \big(-m^{a b}\big) \big) 
+ \overline{\del}_a \big( \big(-m^{\alpha\beta}\big) \Psi^0_\alpha \Psi^0_\beta P^{\gamma a b} \del_\gamma w \big) 
\\
&+ \overline{\del}_a \big( P^{\gamma \alpha\beta} \del_\gamma w \Psi^0_\alpha \Psi^0_\beta  P^{\gamma' a b} \del_{\gamma'} w  \big) 
\Big) \del_s \Gamma^I w \overline{\del}_b \Gamma^I w,
\\
H_3
&= - 2 \Big( \overline{\del}_{a'} \big( s^2 P^{\gamma \alpha\beta} \del_\gamma w \Psi^0_\alpha \Psi^a_\beta \big(-m^{a' b'}\big) \big)
+ \overline{\del}_{a'} \big( s^2 \big(-m^{\alpha\beta}\big) \Psi^0_\alpha \Psi^a_\beta P^{\gamma a' b'} \del_\gamma w \big)
\\
&+ \overline{\del}_{a'} \big( s^2 P^{\gamma \alpha\beta} \del_\gamma w \Psi^0_\alpha \Psi^a_\beta P^{\gamma' a' b'} \del_{\gamma'} w \big) 
\Big) \overline{\del}_a \Gamma^I w \overline{\del}_{b'} \Gamma^I w,
\\
H_4
& = \Big( \overline{\del}_a \big( s^2 P^{\gamma \alpha\beta} \del_\gamma w \Psi^0_\alpha \Psi^a_\beta \big(-m^{a' b'}\big) \big)
+ \overline{\del}_a \big( s^2 \big(-m^{\alpha\beta}\big) \Psi^0_\alpha \Psi^a_\beta P^{\gamma a' b'} \del_\gamma w \big)
\\
&+ \overline{\del}_a \big( s^2 P^{\gamma \alpha\beta} \del_\gamma w \Psi^0_\alpha \Psi^a_\beta P^{\gamma' a' b'} \del_{\gamma'} w \big)
\Big) \overline{\del}_{a'} \Gamma^I w \overline{\del}_{b'} \Gamma^I w,
\endaligned
$$
with also 
$$
\aligned
H_5
&=  - \Big( \overline{\del}_a \big( s^2 P^{\gamma \alpha\beta} \del_\gamma w \Psi^{\gamma'}_\alpha (\overline{\del}_{\gamma'} \Psi^0_\beta) \big(-m^{a b}\big) \big)
+ \overline{\del}_a \big( s^2 \big(-m^{\alpha\beta}\big) \Psi^\gamma_\alpha (\overline{\del}_\gamma \Psi^0_\beta) P^{\gamma' a b} \del_{\gamma'} w \big)
\\
&+ \overline{\del}_a \big( s^2 P^{\gamma \alpha\beta} \del_\gamma w \Psi^{\gamma'}_\alpha (\overline{\del}_{\gamma'} \Psi^0_\beta) P^{\gamma'' a b} \del_{\gamma''} w \big)
\Big) \Gamma^I w \overline{\del}_b \Gamma^I w, 
\\
H_6
&= - s^2 \Big( P^{\gamma \alpha\beta} \del_\gamma w \Psi^{\gamma'}_\alpha (\overline{\del}_{\gamma'} \Psi^0_\beta) \big(-m^{a b}\big)
+ \big(-m^{\alpha\beta}\big) \Psi^\gamma_\alpha (\overline{\del}_\gamma \Psi^0_\beta) P^{\gamma' a b} \del_{\gamma'} w
\\
&+ P^{\gamma \alpha\beta} \del_\gamma w \Psi^{\gamma'}_\alpha (\overline{\del}_{\gamma'} \Psi^0_\beta) P^{\gamma'' a b} \del_{\gamma''} w 
\Big) \overline{\del}_a \Gamma^I w \overline{\del}_b \Gamma^I w,
\endaligned
$$
$$
\aligned
H_7
& = \Big( \overline{\del}_a \big(s P^{\gamma \alpha\beta} \del_\gamma w \Psi^0_\alpha \Psi^0_\beta  \big(-m^{a b}\big) \big) 
+ \overline{\del}_a \big(s \big(-m^{\alpha\beta}\big) \Psi^0_\alpha \Psi^0_\beta  P^{\gamma a b} \del_\gamma w \big) 
\\
&+ \overline{\del}_a \big(s P^{\gamma \alpha\beta} \del_\gamma w \Psi^0_\alpha \Psi^0_\beta  P^{\gamma' a b} \del_{\gamma'} w \big) 
\Big) \Gamma^I w \overline{\del}_b \Gamma^I w, 
\\
H_8
& = s \Big( P^{\gamma \alpha\beta} \del_\gamma w \Psi^0_\alpha \Psi^0_\beta  \big(-m^{a b}\big)
+ \big(-m^{\alpha\beta}\big) \Psi^0_\alpha \Psi^0_\beta  P^{\gamma a b} \del_\gamma w
\\
&+ P^{\gamma \alpha\beta} \del_\gamma w \Psi^0_\alpha \Psi^0_\beta  P^{\gamma' a b} \del_{\gamma'} w
\Big) \overline{\del}_a \Gamma^I w \overline{\del}_b \Gamma^I w.
\endaligned
$$

We will only estimate the representative terms $H_1, H_7$, since the others can be bounded in a similar way. 

\paragraph{--Estimates for the term $H_1$}
For the term $H_1$, recall that $\del_\gamma w = \Psi^{\gamma'}_\gamma \overline{\del}_{\gamma'} w, m^{\alpha\beta} \Psi^0_\alpha \Psi^0_\beta = -1$, and we have
$$
\aligned
&
{1\over 2} \Big\| \del_s \big( s^2 P^{\gamma \alpha\beta} \del_\gamma w \Psi^0_\alpha \Psi^0_\beta  \big(-m^{a b}\big) \big) 
+ \del_s \big( s^2 \big(-m^{\alpha\beta}\big) \Psi^0_\alpha \Psi^0_\beta  P^{\gamma a b} \del_\gamma w \big)
\\
&+ \del_s \big( s^2 P^{\gamma \alpha\beta} \del_\gamma w \Psi^0_\alpha \Psi^0_\beta  P^{\gamma' a b} \del_{\gamma'} w \big)
\Big\|_{L^\infty (\Hcal_s)}
\\
=
& 
{1\over 2} \Big\| \del_s \big( s^2 P^{\gamma \alpha\beta} \Psi^{\gamma'}_\gamma \overline{\del}_{\gamma'} w \Psi^0_\alpha \Psi^0_\beta  \big(-m^{a b}\big) \big) 
+ \del_s \big( s^2 P^{\gamma a b} \del_\gamma w \big)
\\
&+ \del_s \big( s^2 P^{\gamma \alpha\beta} \Psi^{\gamma'}_\gamma \overline{\del}_{\gamma'} w \Psi^0_\alpha \Psi^0_\beta  P^{\gamma'' a b} \del_{\gamma''} w \big)
\Big\|_{L^\infty (\Hcal_s)}.
\endaligned
$$
We proceed to estimate each part, and we get
$$
\aligned
& 
\Big\| \del_s \big( s^2 P^{\gamma \alpha\beta} \Psi^{\gamma'}_\gamma \overline{\del}_{\gamma'} w \Psi^0_\alpha \Psi^0_\beta  \big(-m^{a b}\big) \big) \Big\|_{L^\infty (\Hcal_s)}
\\
&\lesssim
\Big\| \del_s \big( s^2 P^{\gamma \alpha\beta} \Psi^{0}_\gamma \overline{\del}_{0} w \Psi^0_\alpha \Psi^0_\beta  \big) \Big\|_{L^\infty (\Hcal_s)}
+
\Big\| \del_s \big( s^2 P^{\gamma \alpha\beta} \Psi^{a}_\gamma \overline{\del}_{a} w \Psi^0_\alpha \Psi^0_\beta  \big) \Big\|_{L^\infty (\Hcal_s)}
\\
&\lesssim
\Big\| \big|2 s P^{\gamma \alpha\beta} \Psi^{0}_\gamma \overline{\del}_{0} w \Psi^0_\alpha \Psi^0_\beta   \big|
+
\big|  s^2 P^{\gamma \alpha\beta} \Psi^{0}_\gamma \del_s \overline{\del}_{0} w \Psi^0_\alpha \Psi^0_\beta  \big|
+
\big| \del_s \big( \Psi^{0}_\gamma \Psi^0_\alpha \Psi^0_\beta  \big)  s^2 P^{\gamma \alpha\beta} \overline{\del}_{0} w \big| \Big\|_{L^\infty (\Hcal_s)}
\\
&+
\Big\| \big| 2 s P^{\gamma \alpha\beta} \Psi^{a}_\gamma \overline{\del}_{a} w \Psi^0_\alpha \Psi^0_\beta  \big|
+
\big|  s^2 P^{\gamma \alpha\beta} \Psi^{a}_\gamma \del_s \overline{\del}_{a} w \Psi^0_\alpha \Psi^0_\beta  \big|
+
\big| s^2 P^{\gamma \alpha\beta} \Psi^{a}_\gamma \overline{\del}_{a} w \del_s \big( \Psi^0_\alpha \Psi^0_\beta  \big) \big| \Big\|_{L^\infty (\Hcal_s)}.
\endaligned
$$
Successively, the estimates $|P^{\gamma \alpha\beta} \Psi^{0}_\gamma \Psi^0_\alpha \Psi^0_\beta | \lesssim (t/s), |\overline{\del}_{0} w| \lesssim C_1 \eps t^{-1}$ imply
$$
\big\|2 s P^{\gamma \alpha\beta} \Psi^{0}_\gamma \overline{\del}_{0} w \Psi^0_\alpha \Psi^0_\beta   \big\|
\lesssim C_1 \eps, 
$$
and the estimates $|\del_s \overline{\del}_{0} w| \lesssim C_1 \eps t^{-1} s^{-1}$ yield
$$
\big|  s^2 P^{\gamma \alpha\beta} \Psi^{0}_\gamma \del_s \overline{\del}_{0} w \Psi^0_\alpha \Psi^0_\beta  \big|
\lesssim C_1 \eps, 
$$
and the estimates $\del_s \Psi^0_\alpha = - s^{-1} \Psi^0_\alpha + \delta_{0 \alpha} t^{-1}$ give
$$
\big| \del_s \big( \Psi^{0}_\gamma \Psi^0_\alpha \Psi^0_\beta  \big)  s^2 P^{\gamma \alpha\beta} \overline{\del}_{0} w \big\|_{L^\infty (\Hcal_s)}
\lesssim C_1 \eps.
$$
Furthermore, the bounds $|\overline{\del}_{a} w| \lesssim C_1 \eps t^{-1} s^{-1+\delta}$ deduce
$$
\big\| 2 s P^{\gamma \alpha\beta} \Psi^{a}_\gamma \overline{\del}_{a} w \Psi^0_\alpha \Psi^0_\beta  \big\|_{L^\infty (\Hcal_s)}
\lesssim C_1 \eps s^\delta,
$$
and the bounds $|\del_s \overline{\del}_{a} w| \lesssim C_1 \eps t^{-2}$ indicate
$$
\big\|  s^2 P^{\gamma \alpha\beta} \Psi^{a}_\gamma \del_s \overline{\del}_{a} w \Psi^0_\alpha \Psi^0_\beta  \big\|_{L^\infty (\Hcal_s)}
\lesssim C_1 \eps s^\delta,
$$
and the bounds $\del_s \Psi^0_\alpha = - s^{-1} \Psi^0_\alpha + \delta_{0 \alpha} t^{-1}$ lead us to
$$
\big\| s^2 P^{\gamma \alpha\beta} \Psi^{a}_\gamma \overline{\del}_{a} w \del_s \big( \Psi^0_\alpha \Psi^0_\beta  \big) \big\|_{L^\infty (\Hcal_s)}
\lesssim C_1 \eps s^\delta.
$$
Next, we have
$$
\aligned
\big\| \del_s \big( s^2 P^{\gamma a b} \del_\gamma w \big) \big\|_{L^\infty (\Hcal_s)}
\lesssim
\big\| s \del w \big\|_{L^\infty (\Hcal_s)}
+
\big\| s^2 \del_s \del w \big\|_{L^\infty (\Hcal_s)}
\lesssim
C_1 \eps,
\endaligned
$$
in which we used the bounds
$$
|\del w| 
\lesssim C_1 \eps s^{-1},
\qquad
|\del_s \del w | 
\lesssim |(s/t) \del \del w |
\lesssim |(s/t) (t-r)^{-1} \Gamma \del w|
\lesssim C_1 \eps s^{-2}.
$$
Then we turn to 
$$ 
\aligned
&\big\| \del_s \big( s^2 P^{\gamma \alpha\beta} \Psi^{\gamma'}_\gamma \overline{\del}_{\gamma'} w \Psi^0_\alpha \Psi^0_\beta  P^{\gamma'' a b} \del_{\gamma''} w \big) \big\|_{L^\infty (\Hcal_s)}
\\
\lesssim
&\big\| \del_s \big( s^2 P^{\gamma \alpha\beta} \Psi^{\gamma'}_\gamma \overline{\del}_{\gamma'} w \Psi^0_\alpha \Psi^0_\beta \big)  P^{\gamma'' a b} \del_{\gamma''} w  \big\|_{L^\infty (\Hcal_s)}
+
\big\|  s^2 P^{\gamma \alpha\beta} \del_{\gamma} w \Psi^0_\alpha \Psi^0_\beta \del_s \big( P^{\gamma'' a b} \del_{\gamma''} w \big) \big\|_{L^\infty (\Hcal_s)},
\endaligned
$$
and we find
$$
\aligned
&\big\| \del_s \big( s^2 P^{\gamma \alpha\beta} \Psi^{\gamma'}_\gamma \overline{\del}_{\gamma'} w \Psi^0_\alpha \Psi^0_\beta \big)  P^{\gamma'' a b} \del_{\gamma''} w  \big\|_{L^\infty (\Hcal_s)}
\\
\lesssim
& C_1 \eps s^{-1} \big\| \del_s \big( s^2 P^{\gamma \alpha\beta} \Psi^{\gamma'}_\gamma \overline{\del}_{\gamma'} w \Psi^0_\alpha \Psi^0_\beta \big) \big\|_{L^\infty (\Hcal_s)}
\lesssim
\big(C_1 \eps \big)^2 s^{-1+\delta},
\endaligned
$$
and (recall $P^{\gamma \alpha\beta} \del_{\gamma} w \Psi^0_\alpha \Psi^0_\beta \lesssim C_1 \eps s^{-1+\delta}$)
$$
\aligned
\big\|  s^2 P^{\gamma \alpha\beta} \del_{\gamma} w \Psi^0_\alpha \Psi^0_\beta \del_s \big( P^{\gamma'' a b} \del_{\gamma''} w \big) \big\|_{L^\infty (\Hcal_s)}
\lesssim \big( C_1 \eps \big)^2 s^\delta,
\endaligned
$$
which lead to
$$
\big\| \del_s \big( s^2 P^{\gamma \alpha\beta} \Psi^{\gamma'}_\gamma \overline{\del}_{\gamma'} w \Psi^0_\alpha \Psi^0_\beta  P^{\gamma'' a b} \del_{\gamma''} w \big) \big\|_{L^\infty (\Hcal_s)}
\lesssim
\big( C_1 \eps \big)^2 s^\delta.
$$
Thus we obtain
\be
\aligned
\big\| H_1 \|_{L^2_f (\Hcal_s)}
\lesssim C_1 \eps s^\delta \sum_a \big\| \overline{\del}_a \Gamma^I w \big\|^2_{L^2_f (\Hcal_s)}
\lesssim \big( C_1 \eps \big)^3 s^{-2+5\delta},
\endaligned
\ee
which is an integrable quantity provided $\delta \ll 1$.

\paragraph{--Estimates for the term $H_7$}
For the term $H_7$, recall again that $\del_\gamma w = \Psi^{\gamma'}_\gamma \overline{\del}_{\gamma'} w, m^{\alpha\beta} \Psi^0_\alpha \Psi^0_\beta = -1$, and we have
$$
\aligned
&\Big\|
(t/s) \overline{\del}_a \big(s P^{\gamma \alpha\beta} \del_\gamma w \Psi^0_\alpha \Psi^0_\beta  \big(-m^{a b}\big) \big) 
+ (t/s) \overline{\del}_a \big(s \big(-m^{\alpha\beta}\big) \Psi^0_\alpha \Psi^0_\beta  P^{\gamma a b} \del_\gamma w \big) 
\\
+
& (t/s) \overline{\del}_a \big(s P^{\gamma \alpha\beta} \del_\gamma w \Psi^0_\alpha \Psi^0_\beta  P^{\gamma' a b} \del_{\gamma'} w \big) \Big\|_{L^\infty (\Hcal_s)}
\\=
&\Big\|
(t/s) \overline{\del}_a \big(s P^{\gamma \alpha\beta} \Psi^{\gamma'}_\gamma \overline{\del}_{\gamma'} w \Psi^0_\alpha \Psi^0_\beta  \big(-m^{a b}\big) \big) 
+ (t/s) \overline{\del}_a \big(s  P^{\gamma a b} \del_\gamma w \big) 
\\
+
& (t/s) \overline{\del}_a \big(s P^{\gamma \alpha\beta} \Psi^{\gamma'}_\gamma \overline{\del}_{\gamma'} w \Psi^0_\alpha \Psi^0_\beta  P^{\gamma'' a b} \del_{\gamma''} w \big)  \Big\|_{L^\infty (\Hcal_s)}
\endaligned
$$
We start with
$$
\aligned
&\Big\|
(t/s) \overline{\del}_a \big(s P^{\gamma \alpha\beta} \Psi^{\gamma'}_\gamma \overline{\del}_{\gamma'} w \Psi^0_\alpha \Psi^0_\beta  \big(-m^{a b}\big) \big) 
\Big\|_{L^\infty (\Hcal_s)}
\\
\lesssim
&\big\|
(t/s) \overline{\del}_a \big(s P^{\gamma \alpha\beta} \Psi^{0}_\gamma \overline{\del}_{0} w \Psi^0_\alpha \Psi^0_\beta  \big) 
 \big\|_{L^\infty (\Hcal_s)}
+
\big\|
(t/s) \overline{\del}_a \big(s P^{\gamma \alpha\beta} \Psi^{a'}_\gamma \overline{\del}_{a'} w \Psi^0_\alpha \Psi^0_\beta \big) 
 \big\|_{L^\infty (\Hcal_s)}.
\endaligned
$$
On one hand, we have, by recalling $|\overline{\del}_a \big( P^{\gamma \alpha\beta} \Psi^{0}_\gamma  \Psi^0_\alpha \Psi^0_\beta  \big) | \lesssim s^{-1}$, that
$$
\aligned
&\big\|
(t/s) \overline{\del}_a \big(s P^{\gamma \alpha\beta} \Psi^{0}_\gamma \overline{\del}_{0} w \Psi^0_\alpha \Psi^0_\beta  \big) 
 \big\|_{L^\infty (\Hcal_s)}
\\
\lesssim
& \big\|
(t/s) s \overline{\del}_a \big( P^{\gamma \alpha\beta} \Psi^{0}_\gamma  \Psi^0_\alpha \Psi^0_\beta  \big) \overline{\del}_{0} w
 \big\|_{L^\infty (\Hcal_s)}
+
\big\|
(t/s) s P^{\gamma \alpha\beta} \Psi^{0}_\gamma  \overline{\del}_a \overline{\del}_{0} w \Psi^0_\alpha \Psi^0_\beta
 \big\|_{L^\infty (\Hcal_s)} 
\\
\lesssim
& C_1 \eps s^{-1}.  
\endaligned
$$ 
On the other hand, we get
$$
\aligned
&\big\|
(t/s) \overline{\del}_a \big(s P^{\gamma \alpha\beta} \Psi^{a'}_\gamma \overline{\del}_{a'} w \Psi^0_\alpha \Psi^0_\beta \big) 
 \big\|_{L^\infty (\Hcal_s)}
\\
\lesssim
&\big\|
(t/s)  s  \overline{\del}_a \overline{\del}_{a'} w \Psi^0_\alpha \Psi^0_\beta 
 \big\|_{L^\infty (\Hcal_s)}
+  
\big\|
(t/s) s \overline{\del}_{a'} w \overline{\del}_a \big( \Psi^0_\alpha \Psi^0_\beta \big) 
 \big\|_{L^\infty (\Hcal_s)}
\\
\lesssim 
& C_1 \eps s^{-1+\delta}. 
\endaligned
$$
Next, we estimate
$$
\aligned
\big\| (t/s) \overline{\del}_a \big(s  P^{\gamma a b} \del_\gamma w \big)  \big\|_{L^\infty (\Hcal_s)}
\lesssim
\big\| t \overline{\del}_a \del w  \big\|_{L^\infty (\Hcal_s)}
\lesssim
C_1 \eps s^{-1}.
\endaligned
$$
Then, we find
$$
\aligned
&\big\|
(t/s) \overline{\del}_a \big(s P^{\gamma \alpha\beta} \Psi^{\gamma'}_\gamma \overline{\del}_{\gamma'} w \Psi^0_\alpha \Psi^0_\beta  P^{\gamma'' a b} \del_{\gamma''} w \big) \big\|_{L^\infty (\Hcal_s)}
\\
\lesssim
&\big\|
t \overline{\del}_a \big( P^{\gamma \alpha\beta} \del_{\gamma} w \Psi^0_\alpha \Psi^0_\beta \big)  \del w  \big\|_{L^\infty (\Hcal_s)}
+
\big\|
t  P^{\gamma \alpha\beta} \del_{\gamma} w \Psi^0_\alpha \Psi^0_\beta  \overline{\del}_a \del w  \big\|_{L^\infty (\Hcal_s)}
\\
\lesssim
& C_1 \eps s^{-2+\delta}.
\endaligned
$$
Thus we have
\be
\aligned
\big\| H_7 \|_{L^2_f (\Hcal_s)}
\lesssim C_1 \eps s^{-2+4\delta} \sum_a \big\| (s/t) \Gamma^I w \big\|_{L^2_f (\Hcal_s)}  \big\| s \overline{\del}_a \Gamma^I w \big\|_{L^2_f (\Hcal_s)}
\lesssim \big( C_1 \eps \big)^3 s^{-2+4\delta},
\endaligned
\ee
which is integrable as long as $\delta \ll 1$.

Tedious but similar computations allow us to get
\be 
\int_{s_0}^s \int_{\Hcal_{\tau}^*} | M_4| \, dx d\tau
\lesssim  \big( C_1 \eps \big)^3 s^{4\delta}.
\ee
Gathering the above estimates, we finally arrive at the desired conclusion
$$
\widetilde{E}_{con}(\Gamma^I w, s)
\lesssim \eps^2 + (C_1 \eps)^3 s^{4\delta},
\qquad
|I|  \leq N.
$$
\end{proof}



\begin{proof}[Proof of Theorem \ref{thm:main}]
According to the improved estimates in Proposition \ref{prop:improved1} and Proposition \ref{prop:improved2}, if we choose $C_1$ sufficiently large, and $\eps$ sufficiently small (such that $C_1 \eps \ll 1/2$), then we are led to
\bel{eq:improved}
\aligned
E (\Gamma^I w, s)^{1/2}
&\leq {1\over 2} C_1 \eps,
\qquad
&|I| \leq N,
\\
E_{con} (\Gamma^I w, s)^{1/2}
&\leq {1\over 2} C_1 \eps s^\delta,
\qquad
&|I| \leq N-1,
\\
E_{con} (\Gamma^I w, s)^{1/2}
&\leq {1\over 2} C_1 \eps s^{2 \delta},
\qquad
&|I| \leq N.
\endaligned
\ee
This means that $s_1 > s_0$ cannot be of finite value, otherwise, we can extend the solution to a larger hyperbolic time $\widetilde{s}_1 > s_1$, which is thanks to the improved estimates in \eqref{eq:improved}, but this will contradict to the definition of $s_1$ in \eqref{eq:s1}. Thus we conclude that $s_1 = + \infty$.

Then given any time $T \geq t_0 + 1$, we integrate equation \eqref{eq:quasi-div} over the spacetime region (see \cite{PLF-YM-book})
$$
R_0 = \{ (t, x):  t \geq T, t^2 - |x|^2 \leq T^2  \} \bigcap \{(t, x): t \geq |x| + 1 \}
$$
to get
$$
\aligned
& 
\int_{\Hcal_T^*} |\overline{\del} \Gamma^I w|^2 + 2 P^{\gamma \alpha\beta} \del_\gamma w \del_\beta \Gamma^I w \del_t \Gamma^I w n_\alpha - P^{\gamma \alpha \beta} \del_\gamma w \del_\alpha \Gamma^I w \del_\beta \Gamma^I w \, dx
\\
& - \int_{\RR^2}  \big( (\del_t \Gamma^I w)^2 + \sum_a (\del_a \Gamma^I w)^2 \big)
+ 2 P^{\gamma 0 \beta}  \big( \del_\gamma w \del_\beta \Gamma^I w \del_t \Gamma^I w \big)
-  P^{\gamma \alpha\beta}  \big( \del_\gamma w \del_\alpha \Gamma^I w \del_\beta \Gamma^I w \big) \, dx
\\
&=
 \int \int_{R_0} 2 P^{\gamma \alpha\beta} \del_\gamma \del_\alpha w \del_\beta \Gamma^I w \del_t \Gamma^I w
-  P^{\gamma \alpha\beta} \del_t \del_\gamma w \del_\alpha \Gamma^I w \del_\beta \Gamma^I w
\\
&+ 2 \Big(  - \sum_{\substack{|I_1| + |I_2| + d = |I|\\ d\geq 1}} N_d (\Gamma^{I_1} w, \Gamma^{I_2} w) - \sum_{\substack{|I_1| + |I_2|  = |I|\\ |I_2| \leq N-1}} N (\Gamma^{I_1} w, \Gamma^{I_2} w)  \Big) \del_t \Gamma^I w \, dxdt.
\endaligned
$$
We observe that
$$
\aligned
&\int_{\Hcal_T^*} \big| 2 P^{\gamma \alpha\beta} \del_\gamma w \del_\beta \Gamma^I w \del_t \Gamma^I w n_\alpha - P^{\gamma \alpha \beta} \del_\gamma w \del_\alpha \Gamma^I w \del_\beta \Gamma^I w \big| \, dx
\lesssim
C_1 \eps \int_{\Hcal_T^*} |\overline{\del} \Gamma^I w|^2  \, dx,
\\
&\int_{\RR^2} 
\big| 2 P^{\gamma 0 \beta}  \big( \del_\gamma w \del_\beta \Gamma^I w \del_t \Gamma^I w \big)
-  P^{\gamma \alpha\beta}  \big( \del_\gamma w \del_\alpha \Gamma^I w \del_\beta \Gamma^I w \big) \big| \, dx
\lesssim
C_1 \eps \int_{\RR^2}  \big| \del \Gamma^I w \big|^2  \, dx,
\endaligned
$$
since $C_1 \eps \ll 1$, and we thus have
\be 
\aligned
&\| \del \Gamma^I w (T) \|_{L^2(\RR^2)}^2
\\
\lesssim
&\| \overline{\del} \Gamma^I w \|_{L^2_f(\Hcal_T)}^2
+
\Big|\int \int_{R_0} 2 P^{\gamma \alpha\beta} \del_\gamma \del_\alpha w \del_\beta \Gamma^I w \del_t \Gamma^I w
-  P^{\gamma \alpha\beta} \del_t \del_\gamma w \del_\alpha \Gamma^I w \del_\beta \Gamma^I w
\\
+ 
& 2 \Big(  - \sum_{\substack{|I_1| + |I_2| + d = |I|\\ d\geq 1}} N_d (\Gamma^{I_1} w, \Gamma^{I_2} w) - \sum_{\substack{|I_1| + |I_2|  = |I|\\ |I_2| \leq N-1}} N (\Gamma^{I_1} w, \Gamma^{I_2} w)  \Big) \del_t \Gamma^I w \, dxdt \Big|.
\endaligned
\ee
By observing that
$
R_0 \subset \bigcup_{s_0 \leq s \leq T} \Hcal_s,
$
we have
$$
\aligned
& 
\Big|\int \int_{R_0} 2 P^{\gamma \alpha\beta} \del_\gamma \del_\alpha w \del_\beta \Gamma^I w \del_t \Gamma^I w
-  P^{\gamma \alpha\beta} \del_t \del_\gamma w \del_\alpha \Gamma^I w \del_\beta \Gamma^I w
\\
&+ 2 \Big(  - \sum_{\substack{|I_1| + |I_2| + d = |I|\\ d\geq 1}} N_d (\Gamma^{I_1} w, \Gamma^{I_2} w) - \sum_{\substack{|I_1| + |I_2|  = |I|\\ |I_2| \leq N-1}} N (\Gamma^{I_1} w, \Gamma^{I_2} w)  \Big) \del_t \Gamma^I w \, dxdt \Big|
\\
&\lesssim
\int_{s_0}^T \int_{\Hcal_{\tau}^*} (\tau/t) \Big| 2 P^{\gamma \alpha\beta} \del_\gamma \del_\alpha w \del_\beta \Gamma^I w \del_t \Gamma^I w
-  P^{\gamma \alpha\beta} \del_t \del_\gamma w \del_\alpha \Gamma^I w \del_\beta \Gamma^I w
\\
&+ 2 \Big(  - \sum_{\substack{|I_1| + |I_2| + d = |I|\\ d\geq 1}} N_d (\Gamma^{I_1} w, \Gamma^{I_2} w) - \sum_{\substack{|I_1| + |I_2|  = |I|\\ |I_2| \leq N-1}} N (\Gamma^{I_1} w, \Gamma^{I_2} w)  \Big)  \del_t \Gamma^I w \Big| \, dxd\tau
\\
&\lesssim (C_1 \eps)^3,
\endaligned
$$
and we deduce that 
$
\| \del \Gamma^I w \|_{L^2(\RR^2)} \lesssim C_1 \eps$ for all
$|I| \leq N$. 
\end{proof}


\section*{Acknowledgement}

The second author (PLF) is very grateful to Prof. Li Ta-Tsien for giving him the opportunity to visit and work at
the School of Mathematical Sciences, Fudan University. 



\end{document}